\documentclass [12pt,twoside,a4paper]{article}
\usepackage{amsfonts}
\usepackage{amsthm}
\usepackage{amsmath}
\usepackage{amstext}
\usepackage{amssymb}
\usepackage{mathrsfs}
\usepackage{amscd}
\usepackage{xypic}
\usepackage{cite}
\usepackage[numbers,sort&compress]{natbib}
\usepackage{epsf}              
\usepackage{fancybox}          
\usepackage{color}             
\usepackage{fancyhdr}
\usepackage[hang,footnotesize]{caption2}  
\usepackage{enumerate} 
\usepackage{float}
\usepackage{makecell}
\usepackage{subfigure}
\usepackage{graphicx}
\usepackage{epstopdf}
\usepackage{epsfig}
\usepackage[numbers]{natbib}
\usepackage{ragged2e}
\usepackage{graphicx,epsfig,epstopdf,subfigure}
\numberwithin{equation}{section}
\newfont{\aaa}{cmb10 at 19pt}
\newfont{\bbb}{cmb10 at 14pt}
\newtheorem{Case}{Case}[section]

\newtheorem{Subcase}{Subcase}[Case]
\newtheorem{theorem}{Theorem}[section]
\newtheorem{corollary}[theorem]{Corollary}
\newtheorem{conjecture}[theorem]{Conjecture}
\newtheorem{lemma}[theorem]{Lemma}
\newtheorem{claim}[theorem]{Claim}
\newtheorem{definition}[theorem]{Definition}
\newtheorem{proposition}[theorem]{Proposition}
\newtheorem{problem}[theorem]{Problem}
\newtheorem{ques}[theorem]{Question}
\newtheorem{fact}[theorem]{Fact}

\makeatletter

\newcommand{\Rmnum}[1]{\expandafter\@slowromancap\romannumeral #1@}
\makeatother

\newcommand{\beq}{\begin{equation}}
\newcommand{\eeq}{\end{equation}}
\newcommand{\bey}{\begin{eqnarray}}
\newcommand{\eey}{\end{eqnarray}}
\newcommand{\beyy}{\begin{eqnarray*}}
\newcommand{\eeyy}{\end{eqnarray*}}

\bfseries\normalfont

\begin{document}

\title{A Dirac-type theorem for arbitrary Hamiltonian $H$-linked digraphs \thanks{Zhilan Wang and Jin Yan are supported by NNSF of China (No.12071260), and Yangyang Cheng is supported by a PhD studentship of ERC Advanced Grant 883810.}}

\author{Yangyang Cheng\textsuperscript{1}, Zhilan Wang\textsuperscript{2}\thanks{Corresponding author. E-mail address: zhilanwang@mail.sdu.edu.cn.}, Jin Yan\textsuperscript{2}, \unskip\\[2mm]
\textsuperscript{1} Mathematical Institute, University of Oxford, Oxford OX1 2JD, UK
\\
\textsuperscript{2} School of Mathematics, Shandong University, Jinan 250100, China}
\date{}
\maketitle
\begin{abstract}
\noindent Given any digraph $D$ on $n$ vertices, let $\mathcal{P}(D)$ be the family of all directed paths in $D$, and let $H$ be a digraph with the arc set $A(H)=\{a_1, \ldots, a_k\}$.
The digraph $D$ is called arbitrary Hamiltonian $H$-linked if for any injective map $f: V(H)\rightarrow
V(D)$ and any integer set $\mathcal{N}=\{n_1, \ldots, n_k\}$ satisfying that $n_i\geq4$ for each $i\in\{1, \ldots, k\}$, there is a map $g: A(H)\rightarrow \mathcal{P}(D)$ such that for every arc $a_i=uv$, $g(a_i)$ is a directed path from $f(u)$ to $f(v)$ of length $n_i$, and different arcs are mapped into internally vertex-disjoint directed paths in $D$, and $\bigcup_{i\in[k]}V(g(a_i))=V(D)$. Here, the length of a directed path is defined as the number of its arcs.

In this paper, we prove that for any digraph $H$ with $k$ arcs and $\delta(H)\geq1$, there exists a constant $C_0=C_0(k)$ such that if $D$ is a digraph of order $n\geq C_0$ and minimum in- and out-degree at least $n/2+k$, then it is arbitrary Hamiltonian $H$-linked. The lower bound on the minimum in- and out-degree is best possible. We further prove a more general form that allows $k$ to be linear in $n$, while imposing some restrictions on the lengths of the subdivided arcs.
As corollaries, we solved a conjecture of Wang \cite{Wang} for sufficiently large graphs, and partly answered a problem raised by Pavez-Sign\'{e} \cite{Pavez}.

\end{abstract}

\section{Introduction}
Given a (di)graph $G$, let $\mathcal{P}(G)$ denote the family of (directed) paths in $G$. Let $H$ be a fixed (di)graph (possibly
containing loops). An \emph{$H$-subdivision} in $G$ is a pair of maps $f: V(H)\rightarrow
V(G)$ and $g: E(H)\rightarrow \mathcal{P}(G)$ satisfying the following conditions: $(a)$ $f$ is injective, i.e., $f(u)\neq f(v)$ for any distinct vertices $u, v\in V(H)$, and $(b)$ for every edge $uv\in E(H)$, $g(uv)$ is a (directed) path connecting $f(u)$ and $f(v)$, and paths corresponding to different edges are internally vertex-disjoint in $G$. For any (directed) path $P$, its \emph{length} is defined as the number of (directed) edges in $P$. A (di)graph $G$ is \emph{$H$-linked} if every injective mapping $f: V(H)\rightarrow V(G)$ can be extended to an $H$-subdivision in $G$. Furthermore, $G$ is called \emph{Hamiltonian $H$-linked} if every injective mapping $f: V(H)\rightarrow V(G)$ can be extended to a spanning $H$-subdivision in $G$, meaning that the union of the vertex sets of the paths in the subdivision covers all vertices of $G$, i.e., $\bigcup_{uv\in E(H)}V(g(uv))=V(G)$. This framework generalizes the concepts of subdivisions and linkage in (di)graphs, providing a foundation for studying structural and extremal properties of graphs and digraphs.

\smallskip

Researchers have been particularly intrigued by the question of what degree conditions ensure that a graph $G$ is $H$-linked for any fixed graph $H$. In $2005$, Kostochka and Yu \cite{Kostochka} proved that for a simple graph $H$ with $k$ edges and minimum degree $\delta(H)\geq2$, if $G$ is a graph of order $n\geq5k+6$, and $\delta(G)\geq\frac{n+k-2}{2}$, then $G$ is $H$-linked. Moreover, under the same degree condition, $G$ is also Hamiltonian $H$-linked. Since then, their theorem has inspired numerous generalizations and extensions, with significant contributions documented in \cite{Ferrara, Gould, Kostochka1, Kostochka2}. These works have further explored and refined the degree conditions and structural properties required for $H$-linkage and Hamiltonian
$H$-linkage in graphs. For other degree conditions like Ore-type conditions, recently, Coll, Magnant, and Nowbandegani \cite{Coll} proved that there exists a positive integer $n_0$ such that for any integer set $\mathcal{N}=\{n_1, \ldots, n_k\}$ with $n_i\geq n_0$ for all $i\in\{1, \ldots, k\}$, and for any graph $H$ with $k$ edges and $\delta(H)\geq1$, every graph $G$ of order $n$ with $\sigma_2(G)\geq n+2k-1$ is Hamiltonian $(\mathcal{N}H)$-linked (i.e., $G$ contains a spanning $H$-subdivision, in which the paths have lengths $n_1, \ldots, n_k$, respectively), where $\sigma_2(G)=\min \{d(u)+d(v) \mid uv \notin E(G)\}$.

\smallskip

Let $D=(V, A)$ be a digraph on $n$ vertices, and let $H$ be a digraph with the arc set $A(H)=\{a_1, \ldots, a_k\}$. Given any integer set $\mathcal{N}=\{n_1, \ldots, n_k\}$, an $H$-subdivision $(f, g)$ is \emph{Hamiltonian $(\mathcal{N}H)$-subdivision} if for each $i\in\{1, \ldots, k\}$, the length of the path $g(a_i)$ is $n_i$, and $\bigcup_{i\in[k]}V(g(a_i))=V(D)$. In particular, a digraph $D$ is called \emph{arbitrary Hamiltonian $H$-linked}
if for any injective map $f: V(H)\rightarrow
V(D)$ and any integer set $\mathcal{N}=\{n_1, \ldots, n_k\}$ satisfying that $n_i\geq4$ for each $i\in\{1, \ldots, k\}$, every map $f: V(H)\rightarrow V(D)$ can be extended to a Hamiltonian $(\mathcal{N}H)$-subdivision.
Furthermore, $D$ is called $(\alpha, \beta)$-\emph{arbitrary Hamiltonian $H$-linked}
if there exists constants $\alpha, \beta\in(0, 1)$ such that for any integer set $\mathcal{N}=\{n_1, \ldots, n_k\}$ satisfying that $n_i\geq4$ for each $i\in\{1, \ldots, k\}$ and $\sum_{n_i<\alpha n}n_i\leq\beta n$, every map $f: V(H)\rightarrow V(D)$ can be extended to a Hamiltonian $(\mathcal{N}H)$-subdivision.

\smallskip

We define the \emph{minimum semi-degree} of $D$ as $\delta^0(D)=\min\{\delta^+(D), \delta^-(D)\}$, where $\delta^+(D)$ and $\delta^-(D)$ denote the minimum out-degree and in-degree of $D$, respectively. The \emph{minimum degree} is defined as $\delta(D)=\min_{x\in V}\{d(x): d(x)=d^+(x)+d^-(x)\}$.

\smallskip

In this paper, we investigate the minimum semi-degree condition that ensures $D$ to be arbitrary Hamiltonian $H$-linked. In fact, we prove the following result:
\begin{theorem}\label{song3}
Let $H$ be a digraph with $k$ arcs and $\delta(H)\geq1$. There exists a constant $C_0$ depending on $H$ such that if $D$ is a digraph of order $n\geq C_0$ and $\delta^0(D)\geq n/2+k$, then $D$ is arbitrary Hamiltonian $H$-linked.
\end{theorem}
We actually proved a more general form under the condition that the short lengths in the subdivision of $H$ are not too many:
\begin{theorem}\label{song2}
There exist constants $\alpha_0, \beta_0 \in (0,1)$ such that for any $\alpha\in (0,\alpha_0]$ and $\beta\in (0,\beta_0]$, there exists a constant $C_0>0$ such that if $H$ is a digraph with $k$ arcs and $\delta(H)\geq1$ and $D$ is a digraph of order $n\geq C_0k$ with $\delta^0(D)\geq n/2+k$, then $D$ is $(\alpha, \beta)$-arbitrary Hamiltonian $H$-linked.
\end{theorem}

In fact, let $\alpha_0$, $\beta_0$ be the constants defined in Theorem \ref{song2}. Choose $\alpha=\min \{\alpha, \beta_0/k\}$ and $\beta=\beta_0$, where $k$ is supposed to be a constant, and let $C_0$ be the constant defined in Theorem \ref{song2} that depends on $\alpha, \beta$. Under the condition of Theorem \ref{song3}, the sum $\sum_{n_i<\alpha n}n_i\leq \alpha kn\leq \beta_0 n$. Thus, by Theorem \ref{song2}, when $n\geq C_0k$, $D$ is arbitrary Hamiltonian $H$-linked.

In the proof of Theorem \ref{song2}, we use the standard absorption method, which was first introduced by R\"{o}dl, Ruci\'{n}ski and Szemer\'{e}di \cite{Rodl}, as well as the stability method. We need to adapt these ideas to the linkage of digraphs instead of tight cycles in hypergraphs. Generally speaking, the proof of Theorem \ref{song2} is divided into two parts. In the first part of the proof, we assume that the digraph $D$ is not close to one extremal case. Under this assumption, we apply the absorption method in a standard way to prove the existence of an arbitrary $(\alpha, \beta)$-Hamiltonian $H$-subdivision. In the second part of the proof, we consider the case where $D$ is close to the defined extremal case. Here, we analyze the structure of $D$ case by case using various structural methods. This part of the proof is more intricate and is detailed in Section 3.2.

The following remarks show that the degree condition and length condition in Theorem \ref{song2} are both best possible.

\smallskip

\noindent \emph{Remark $1$.} It is not meaningful to seek a condition based solely on the minimum out-degree (or similarly, minimum in-degree) of a digraph $D$ to ensure that $D$ is $H$-linked. To illustrate this, consider the following construction: let $D$ be the digraph obtained from a complete digraph $D_0$ of order $n-1$ by adding a new vertex $x$ that sends an arc to every vertex in $D_0$. Here, a complete digraph is defined as a digraph containing all possible arcs. Clearly, $\delta^+(D)\geq n-2$, but $D$ is not even $\overrightarrow{K_2}$-linked. This demonstrates that high minimum out-degree (or in-degree) alone is insufficient to guarantee $H$-linkage.

\medskip

$2$. The minimum semi-degree condition in Theorem \ref{song2} is best possible. We present a counterexample inspired by the works of K\"{u}hn and Osthus \cite{20081}, and K\"{u}hn, Osthus and Young \cite{20082}. Let $D$ be a digraph consisting of two complete digraphs $Q_1$ and $Q_2$, each of order $n/2+k$, which share exactly $2k$ common vertices. By calculating the semi-degrees of vertices in $V(Q_1\setminus Q_2)\cup V(Q_2\setminus Q_1)$, we observe that $\delta^0(D)=n/2+k-1$. Let $V(Q_1\cap Q_2)=\{x_1, \ldots, x_k, y_1, \ldots, y_k\}$ and let $H$ be the digraph defined by the arcs $x_1y_1\cup\cdots\cup x_ky_k$.  In this case, $D$ is not arbitrary Hamiltonian $H$-linked because the induced subdigraph $D[V(D-H)\cup\{x_1, y_1\}]$ does not contain a path of length more than $n/2-k+1$ from $x_1$ to $y_1$. This construction confirms the tightness of the minimum semi-degree condition in Theorem \ref{song2}.

\medskip

$3$. The condition $n_i\geq4$ for all $1\leq i\leq k$ is necessary in the following sense. Let $k$ and $n$ be integers such that $k\geq5$ and $l=\frac{n+3k-1}{2(3k-2)}$ is an integer greater than or equal to $2$. Consider the digraph $D=\frac{n-3k+1}{2}K_1+l\overleftrightarrow{K}_{3k-2}$, where the `$+$' denotes the addition of all possible arcs between two subdigraphs (or vertex sets), $\frac{n-3k+1}{2}K_1$ represents $\frac{n-3k+1}{2}$ isolated vertices, and $l\overleftrightarrow{K}_{3k-2}$ represents
$l$ disjoint copies of the complete digraph $\overleftrightarrow{K}_{3k-2}$. It can be verified that  $D$ has order $n$ and $\delta^0(D)=\frac{n+3k-5}{2}\geq\frac{n+2k}{2}$. However, if $H$ consists of $k$ disjoint arcs within one of the component of $l\overleftrightarrow{K}_{3k-2}$, then $D$ does not contain the Hamiltonian $H$-subdivision described in Theorem \ref{song2} for the case when $n_1=\cdots=n_{k-1}=3$ and $n_k=n-4k+3$.

\smallskip

Theorem \ref{song2} refines and extends several earlier findings with the minimum degrees differ by only $1$. K\"{u}hn and Osthus \cite{20081} proved that a minimum semi-degree of $n/2+k-1$ is sufficient to ensure a sufficiently large digraph $D$ is $k\overrightarrow{K_2}$-linked, where $\overrightarrow{K_2}$ represents a single arc and $k\overrightarrow{K_2}$ denotes the union of $k$ vertex-disjoint arcs. Later, K\"{u}hn, Osthus and Young \cite{20082} showed that under the same minimum semi-degree condition, $D$ is also Hamiltonian $k\overrightarrow{K_2}$-linked. Additionally, in \cite{Jacobson}, Ferrara, Jacobson and Pfender established  the minimum semi-degree condition guaranteeing that a digraph is $H$-linked for any multidigraph $H$. In comparison, our Theorem \ref{song2} requires a minimum semi-degree of $n/2+k$, which is $1$ more than the $n/2+k-1$ condition in the aforementioned works. As highlighted in the remarks above, this difference is necessary and cannot be improved to $n/2+k-1$ due to the existence of extremal cases.

\smallskip

Wang \cite{Wang} proposed the following conjecture about disjoint cycles passing through prescribed edges under degree condition:
\begin{conjecture}
If $G$ is a graph of order $n$ with minimum degree at least $n/2+k$, then for any $k$ disjoint edges $e_1, \ldots, e_k$ in $G$ and for any integer partition $n=n_1+\cdots+n_k$ with $n_i\geq5$ for each $i\in\{1, \ldots, k\}$, $G$ has $k$ vertex-disjoint cycles $C_1, \ldots, C_k$ of orders $n_1, \ldots, n_k$, respectively, such that $C_i$ passes through $e_i$ for all $1\leq i\leq k$.
\end{conjecture}
By replacing edges of $G$ with two arcs in both directions, it is straightforward to verify the following corollary holds directly from Theorem \ref{song2}, and thus gives an affirmative answer to Wang's conjecture when $n$ is sufficiently large compared to $k$.
\begin{corollary}
For every integer $k\geq 2$, there exists an integer $C_0=C_0(k)$ such that if $G$ is a graph of order $n\geq C_0$ with minimum degree at least $n/2+k$, then for any $k$ independent edges $e_1, \ldots, e_k$ in $G$ and for any integer partition $n=n_1+\cdots+n_k$ satisfying that $n_i\geq5$ for each $i\in\{1, \ldots, k\}$, $G$ has $k$ vertex-disjoint cycles $C_1, \ldots, C_k$ of orders $n_1, \ldots, n_k$, respectively, such that $C_i$ passes through $e_i$ for all $1\leq i\leq k$.
\end{corollary}
Pavez-Sign\'{e} \cite{Pavez} proposed the following question regarding $H$-subdivision:
\begin{ques}
For every $\varepsilon>0$, does there exist a positive constant $C_0>0$ such that for all $C\geq C_0$ and any integer $k$ the following holds? Let $H$ be a graph with $k$ edges and $\delta(H)\geq1$. If $G$ is a graph on $n\geq Ck$ vertices and minimum degree at least $(1/2+\varepsilon)n$, then it contains a spanning $H$-subdivision, where all the paths in the subdivision have nearly the same length.
\end{ques}

By replacing edges of $G$ with two arcs in both directions, the following corollary also holds directly from Theorem \ref{song2}, and thus gives a partly answer to Pavez-Sign\'{e}'s problem when $k$ is fixed.
\begin{corollary}
Let $H$ be any graph with $k$ edges and $\delta(H)\geq1$. There exists a positive integer $C_0=C_0(k)$ such that if $G$ is a graph of order $n\geq C_0$ and $\delta(G)\geq n/2+k$, then $G$ is arbitrary Hamiltonian $H$-linked.
\end{corollary}
A digraph $D$ is \emph{$k$-ordered} if $|V(D)|\geq k$ and for every sequence $s_1, \ldots, s_k$ of distinct vertices in $D$, there exists a cycle that encounters $s_1, \ldots, s_k$ in this order. Furthermore, $D$ is said to be \emph{arbitrary $k$-ordered Hamiltonian} if this cycle is Hamiltonian, and for any integer set $\{n_1, \ldots, n_s\}$ where $n_i\geq 4$ for $1\leq i\leq s$, the length of the path on this cycle connecting vertex $s_i$ to $s_{i+1}$ is $n_i$ for each $i\in\{1, \ldots, s\}$. The definitions of $k$-arc-ordered digraphs and arbitrary $k$-arc-ordered Hamiltonian digraphs are analogous.

K\"{u}hn and Osthus \cite{20081} proved that there exists some constant $c$ such that for any $k\geq2$, every digraph $D$ of order $n\geq ck^3$ is $k$-ordered if $\delta^0(D)\geq(n+k)/2-1$; and is $k$-arc-ordered if $\delta^0(D)\geq n/2+k-1$. Also, K\"{u}hn, Osthus and Young \cite{20082} showed that for every $k\geq3$ there is an integer $n_0=n_0(k)$ such that every digraph $D$ on $n\geq n_0$ vertices with $\delta^0(D)\geq\lceil(n+k)/2\rceil-1$ is $k$-ordered Hamiltonian; and under the same assumptions, if $\delta^0(D)\geq\lceil n/2\rceil+k-1$ then $D$ is $k$-arc-ordered Hamiltonian. We give the following corollary of Theorem \ref{song2}.
\begin{corollary}\label{y1}
For any integer $k\geq2$, there exists an integer $C_0=C_0(k)$ such that every digraph $D$ of order $n\geq C_0$ with $\delta^0(D)\geq n/2+k$ is arbitrary $k$-arc-ordered (and $k$-ordered) Hamiltonian.
\end{corollary}
Note that by Remark 2 above, the lower bound $n/2+k$ of Corollary \ref{y1} cannot be improved to $n/2+k-1$ due to the existence of counterexamples.

If the digraph $H$ consists of $k$ disjoint loops, then Theorem \ref{song2} leads to the following conclusion, which provides a special case of El-Zahar's conjecture \cite{El} in the directed version when the number of cycles is bounded, in a stronger form.
\begin{corollary}
For every positive integer $k$, there exists an integer $C_0=C_0(k)$ such that, if $D$ is a digraph of order $n\geq C_0$ and $\delta^0(D)\geq n/2+k$, and $S=\{x_1, \ldots, x_k\}$ be any vertex set of $D$. Then for any integer partition $n=n_1+\cdots+n_k$ satisfying that $n_i\geq4$ for each $i\in\{1, \ldots, k\}$, $D$ contains $k$ vertex-disjoint cycles $C_1, \ldots, C_k$ of length $n_1, \ldots, n_k$, respectively, such that $V(C_i)\cap S=\{x_i\}$ for all $i\in\{1, \ldots, k\}$.
\end{corollary}
\noindent \textbf{Organization.} The rest of the paper is organized as follows. In Section $2$, we begin by presenting relevant definitions and notations. We then  provide a sketch of the proof of Theorem \ref{song2}. Moving on to Section $3$, we present the detailed proof of Theorem \ref{song2}. In Subsection $3.1$,
we firstly introduce some key lemmas, namely the Connecting Lemma, Absorbing Lemma and Path-Covering Lemma, which are utilized to prove Theorem \ref{song2} when the digraph $D$ does not satisfy the extremal condition discussed in Section $2$. Secondly, we provide the proof of Theorem \ref{song2} for the case when $D$ does not satisfy the extremal condition. In Subsection $3.2$, we first identify one extremal case that $D$ belongs to when it satisfies the extremal condition, and then we will prove that Theorem \ref{song2} holds in this case. Finally, Section $4$ contains some concluding remarks to wrap up the paper.
\section{Preparations for Theorem \ref{song2}}
\subsection{Definitions and notations}
For notations not defined in this paper, we refer the readers to \cite{Bang-Jensen3}.
Let $D=(V, A)$ be a digraph. The cardinality of a vertex set $X\subseteq V$ is denoted by $|X|$,  and we call $X$ an \emph{$i$-set} if $|X|=i$. The \emph{out-neighbourhood} (resp., \emph{in-neighbourhood}) of a vertex $v$ in $D$ is defined as $N^{+}(v)=\{u: vu\in A\}$ (resp., $N^{-}(v)=\{w: wv\in A\}$). The \emph{out-degree} (resp., \emph{in-degree}) of $v$ in $D$, denoted by $d^+(v)$ (resp. $d^-(v)$), is the cardinality of $N^{+}(v)$ (resp., $N^{-}(v)$), that is, $d^{+}(v)=|N^{+}(v)|$ (resp., $d^{-}(v)=|N^{-}(v)|$). The \emph{minimum out-degree} $\delta^+(D)=\min\{d^{+}(v): v\in V\}$ and the \emph{minimum in-degree} $\delta^-(D)=\min\{d^{-}(v): v\in V\}$.

\smallskip

For any $X\subseteq V$ and $\sigma\in\{-, +\}$, we define $N^\sigma(u, X)=N^\sigma(u)\cap X$ and $d^\sigma_X(u)=|N^\sigma(u, X)|$ for any vertex $u$ in $V$, and $\delta^0_X(u)=\min\{d^+_X(u), d^-_X(u)\}$. The subdigraph of $D$ induced by $X$ is denoted as $D[X]$. Let $D-X=D[V\setminus X]$ and $\overline{X}=V\setminus X$. For another vertex set $Y$ that is not necessarily disjoint from $X$, we use $e^+(X, Y)$ to represent the number of arcs from $X$ to $Y$. In particular, $e(X)$ represents the number of arcs in $D[X]$. In this paper, we also abbreviate the bipartite digraph $D[X, Y]$ as $(X, Y)$.

\smallskip

A \emph{$k$-path} refers to a path with $k$ vertices. We often represent the $k$-path $P$ as $v_1\cdots v_k$, where $V(P)=\{v_1, \ldots, v_k\}$, and call $v_1$ and $v_k$ the \emph{initial} and the \emph{terminal} of $P$, respectively. Furthermore, for two disjoint vertex sets $X$ and $Y$ in $V$, if the initial and the terminal of $P$ belongs to $X$ and $Y$, respectively, then we say that $P$ is an $(X, Y)$-path. In particular, we write an $X$-path instead of $(X, X)$-path if $X=Y$. Additionally, we say an $(X, Y)$-path $P$ is \emph{minimal} if there is no $(X, Y)$-path $P^\prime$ with $|V(P^\prime)|<|V(P)|$. All paths in digraphs refer to directed paths, and we use the term \emph{disjoint} instead of vertex-disjoint for simplicity.

\smallskip

For a vertex pair $(u, v)$ (possibly, $u=v$), we say that a $4$-path $z_1z_2z_3z_4$ \emph{absorbs} $(u, v)$ if $z_2u, vz_3\in A$; and a $4$-path is called a \emph{absorber} for $(u, v)$ if it absorbs $(u, v)$. This terminology reflects the fact that the $4$-path $z_1z_2z_3z_4$ can be extended by absorbing a path with the initial $u$ and the terminal $v$, resulting in a longer path with the same set of end-vertices. For two paths $P=a\cdots b$ and $Q=b\cdots d$ with $V(P)\cap V(Q)=\{b\}$, we denote the concatenated path as $P\circ Q$. This definition can be extended naturally to more than two paths.

\smallskip

For a positive integer $t$, we simply write $\{1, \ldots , t\}$ as $[t]$. Throughout this paper, the notation $0<\beta\ll\alpha$ is used to make clear that $\beta$ can be selected to be sufficiently small corresponding to $\alpha$ so that all calculations required in our proof are valid. For the real numbers $a$ and $b$, we use $a\pm b$ to represent an unspecified real number in the interval $[a-b, a+b]$.

\smallskip

To summarize this subsection, we provide the following extremal condition for a constant $\varepsilon^\prime$, where $0<\varepsilon^\prime\ll1$. In particular, we say the digraph $D$ is \emph{stable} if $D$ does not satisfy the following extremal condition (\textbf{EC}).

\smallskip

\noindent \textbf{Extremal Condition (EC) with parameter $\varepsilon^\prime$:} Let $D$ be a digraph of order $n$. There exist two (not necessarily disjoint) vertex sets $U_1$ and $U_2$ in $D$ with $|U_i|\geq(1/2-\varepsilon^\prime)n$ for every $i\in[2]$ such that $e^+(U_1, U_2)\leq(\varepsilon^\prime n)^2$.
\subsection{Overview of the proof of Theorem \ref{song2}}
Let $H$ be a digraph with $k$ arcs and $\delta(H)\geq1$, and let $D$ be a digraph of order $n\geq C_0k$ with $\delta^0(D)\geq n/2+k$, as described in Theorem \ref{song2}. The proof of Theorem \ref{song2} utilizes the stability method, which is divided into two main cases:

The \textbf{extremal case}, when the digraph $D$ is not stable;

The \textbf{non-extremal case}, when $D$ is stable.

\smallskip

\textbf{Non-extremal case}

For the non-extremal case, we divide the proof into the following three steps:

\textbf{Step 1. Prove the Connecting lemma.} The Connecting Lemma (referred to as Lemma \ref{s1} in Subsection $3.1$ below) asserts that any two distinct vertices in $D$ can be connected by a short path.

\textbf{Step 2. Find an $H$-linked subdigraph (Absorbing Lemma).} By utilizing the Connecting Lemma and the probabilistic method, we will construct an \emph{absorbing} subdigraph $H^\prime$ that is $H$-linked and possesses the remarkable property that for every vertex pair $(u, v)$ in $D-H^\prime$, any `long' subdivided path of $H^\prime$ contain at least one absorber for $(u, v)$.

\textbf{Step 3. Path-Covering Lemma.} The Path-Covering Lemma (Lemma \ref{Path-Cover Lemma} in Subsection $3.1$) implies that we can use a limited number of disjoint paths, of any lengths, to cover all vertices of $D-H^\prime$.

Consequently, by using the absorbing property of $H^\prime$, we will absorb these disjoint paths of suitable lengths into $H^\prime$ to obtain the desired arbitrary $(\alpha, \beta)$-Hamiltonian $H$-linked subdigraph. This strongly suggests that the main theorem holds.

\smallskip

\textbf{Extremal case}

For the extremal case, we employ the traditional structural analysis method to demonstrate that the main theorem holds. Or equivalently, we will show that

\textbf{Step 4.} The digraph $D$ falls into the Extremal Case $1$ (\textbf{EC1}), which is defined in Subsection $3.2$ below. We will establish that $D$ is $(\alpha, \beta)$-arbitrary Hamiltonian $H$-linked in this case.

\smallskip

In particular, our approach to proving the Path-Covering Lemma relies on a directed version of expanders known as robust outexpanders. This concept was explicitly introduced by K\"{u}hn, Osthus and Treglown \cite{Kuhn01}. The notion of robust expansion has played a crucial role in the solution of several conjectures related to the packing of Hamiltonian cycles and paths in (di)graphs. For more recent applications of the theory of robust outexpanders, we recommend interested readers to refer to \cite{Kelly,Keevash1,Kuhn0,Kuhn01,Kuhn1,Staden}.
\section{Proof of Theorem \ref{song2}}
\subsection{Non-extremal case}
Let $H$ be a digraph with $k$ arcs and $\delta(H)\geq1$. In this section, all statements assume that $D$ is a digraph on $n\geq C_0k$ vertices and satisfies $\delta^0(D)\geq n/2+k$, as stated in Theorem \ref{song2}.
Additionally, we suppose that $D$ is stable, and let $\alpha, \beta, \varepsilon, \varepsilon^\prime$, $\varepsilon_1$ and $\gamma$ be parameters chosen such that $0<1/C_0< \alpha, \beta, \gamma\ll\varepsilon^\prime\ll\varepsilon_1\ll\varepsilon\ll1$.
\subsubsection{Connecting and absorbing}
The following lemma asserts that any two distinct vertices can be connected by a short directed path in $D$.
\begin{lemma}\emph{(}Connecting Lemma\emph{)}\label{s1}
Let $D$ be a digraph on $n\geq C_0k$ vertices with $\delta^0(D)\geq n/2+k$, and $D$ is stable. The parameters $\gamma$ and $\varepsilon^\prime$ satisfy $0<1/C_0\ll\gamma\ll\varepsilon^\prime\ll1$. Let $P_1$ and $P_2$ be two disjoint paths of length at most $3$ in $D$. Then there exists a $q$-path with $q\leq4$ in $D$ that connects the paths $P_1$ and $P_2$. Furthermore, this conclusion still holds even if at most $\gamma n$ vertices are forbidden to be used on this connecting path.
\end{lemma}
\begin{proof}
Assume that the initial arc and the terminal arc of $P_1$ and $P_2$ are $ab$, and $cd$, respectively. There is nothing to prove if $bc\in A(D)$. So we assume that $bc\notin A(D)$.
Let $U$ be a vertex subset of $D$ with $|U|\leq \gamma n$, and define $D_0=D[V(D)\setminus(U\cup V(P_1-b)\cup V(P_2-c))]$. If $N^+_{D_0}(b)\cap N^-_{D_0}(c)\neq\emptyset$, then there exists a vertex $x\in N^+_{D_0}(b)\cap N^-_{D_0}(c)$, and the desired connecting path is $P=bxc$. Otherwise, note that  $$|N^+_{D_0}(b)|, |N^-_{D_0}(c)|\geq\delta^0(D)-(\gamma n+6)\geq n/2+k-(\gamma n+6)\geq(1/2-\varepsilon^\prime)n.$$ Then we may deduce that $e^+(N^+_{D_0}(b), N^-_{D_0}(c))>(\varepsilon^\prime n)^2$, since $D$ does not satisfy the extremal condition with $(U_1, U_2)_{\textbf{EC}}=(N^+_{D_0}(b), N^-_{D_0}(c))$. This suggests that there exists an arc $xy$ from $N^+_{D_0}(b)$ to $N^-_{D_0}(c)$, and the path $P=bxyc$ connects paths $P_1$ and $P_2$. Thus, the lemma is proved.
\end{proof}

\smallskip

We introduce the following two standard probabilistic tools:
\begin{lemma}\label{Chernoff} \cite{Janson}
$(i)$ Chernoff's inequality: Let $X$ be a sum of independent binomial random variables with expectation $\mathbb{E}X$, and let $a$ be any real number with $0<a<3/2$. Then $$\mathbb{P}(|X-\mathbb{E}X|>a\mathbb{E}X)<2e^{-\frac{a^2}{3}\mathbb{E}X}.$$

\noindent $(ii)$ Markov's inequality: If $X$ is a non-negative integer valued random variable with the expectation $\mathbb{E}X$, then for any $a>0$, $$\mathbb{P}(X\geq a)\leq\frac{\mathbb{E}X}{a}.$$
\end{lemma}

Let $H$ be a digraph with $k$ arcs and $\delta(H)\geq1$. In the following, we always suppose that $V(H)=\{v_1, \ldots, v_{|V(H)|}\}$. For convenience, let $W=\{f(v_1), \ldots, f(v_{|V(H)|})\}$, and define $f(v_i):=v_i$ for each $i\in[|V(H)|]$. 

For any vertex pair $(u, v)$ (possibly $u=v$) in $D$, we denote by $\mathcal{A}_{uv}$ the family of all $4$-paths absorbing $(u, v)$. Then, we can conclude that for any vertex pair of $D$, there are at least $\gamma n^4$ different $4$-paths to absorb it.
\begin{fact}\label{qqqw}
Let $D^\prime$ be a digraph on $n\geq C_0k$ vertices with $\delta^0(D^\prime)\geq n/2-k$, and $D^\prime$ is stable. The parameters $\gamma$ and $\varepsilon^\prime$ satisfy $0<1/C_0\ll\gamma\ll\varepsilon^\prime\ll1$. 
Then for any vertex pair $(u, v)$, there are at least $\gamma n^4$ $4$-paths absorbing $(u, v)$ in $D^\prime$, that is, $|\mathcal{A}_{uv}|\geq\gamma n^4$.
\end{fact}
\begin{proof}
Let $U_1=N_{D^\prime}^-(u)$ and $U_2=N_{D^\prime}^+(v)$. By the minimum semi-degree condition of $D^\prime$, we have $|U_i|\geq n/2-k$ for every $i\in[2]$. Since $D^\prime$ is stable, we obtain that $e^+(U_1, U_2)>(\varepsilon^\prime n)^2$. Furthermore, by the lower bound of $\delta^0(D^\prime)$, we deduce that for any given arc $z_1z_2$ with the vertex $z_1\in U_1$ and $z_2\in U_2$, the following holds: $$|N^-_{D^\prime}(z_1)\setminus\{u, v, z_2\}|\geq n/2-k-3\ \mbox{and}\  |N^+_{D^\prime}(z_2)\setminus\{u, v, z_1\}|\geq n/2-k-3.$$ This implies that the number of $4$-paths $z_0z_1z_2z_3$ with $z_0\in N^-_{D^\prime}(z_1)\setminus\{u, v, z_2\}$ and $z_3\in N^+_{D^\prime}(z_2)\setminus\{u, v, z_0, z_1\}$ absorbing $(u, v)$ is at least
\begin{equation*}
\begin{split}
 (\varepsilon^\prime n)^2\cdot(n/2-k-3)\cdot(n/2-k-4)\geq(\varepsilon^\prime)^2n^4/5\geq \gamma n^4,
\end{split}
\end{equation*}
where the last inequality follows from the fact that $\gamma \ll\varepsilon^\prime$.
\end{proof}
Prior to presenting the absorption lemma, we introduce the following preparatory lemma. Lemma \ref{absor1} establishes that in a digraph $D^\prime$ with sufficiently large minimum semi-degree, there exists a small family $\mathcal{F}$ of disjoint absorbers, satisfying that for any vertex pair $(u, v)$ in $D^\prime$, there are enough absorbers in $\mathcal{F}$ to absorb $(u, v)$. Lemma \ref{absor1} further demonstrates how to partition $\mathcal{F}$ into $l$ disjoint subsets $\mathcal{F}_1, \ldots, \mathcal{F}_l$ while preserving its absorption capacity. Specifically, for each subset $\mathcal{F}_i$, by utilizing vertices in $V(D^\prime-\mathcal{F})$ and the extremal condition (\textbf{EC}), there exists a path $L_i$ that covers all absorbers in $\mathcal{F}_i$.


\begin{lemma}\label{absor1}
Let $D^\prime$ be a digraph on $n\geq C_0k$ vertices with $\delta^0(D^\prime)\geq n/2-k$, and $D^\prime$ is stable. Suppose that $\gamma, \varepsilon^\prime, \lambda
$ are parameters satisfying $0<1/C_0\ll\gamma, \lambda
\ll\varepsilon^\prime\ll1$, and let $l$ be a positive integer with $l\leq k$.  
Then there exists a family $\mathcal{F}$ of at most $\gamma n$ disjoint $4$-paths
in $D^\prime$ such that

$(i)$ for every vertex pair $(u, v)$, we have
$|\mathcal{A}_{uv}\cap\mathcal{F}|\geq\gamma^2n$;

$(ii)$ for a partition $|\mathcal{F}|=f_1+\cdots+f_l$ with $\lambda|\mathcal{F}|<f_i<(1-\lambda)|\mathcal{F}|$ for each $i\in[l]$, there exists a partition $\mathcal{F}=\mathcal{F}_1\cup \cdots\cup\mathcal{F}_l$ with $\mathcal{F}_i=\{F_{i, 1}, \ldots, F_{i, f_i}\}$ satisfying:

$1.$ For any vertex pair $(u, v)$ of $D-V(\mathcal{F})$ and any $i\in[l]$, $\mathcal{F}_i$ contains at least one absorber for $(u, v)$;

$2.$ There exist disjoint paths $L_1, \ldots, L_l$ in $D^\prime$ with each $L_i$ structured as $L_i=F_{i, 1}\circ P_{i, 1}\circ F_{i, 2}\cdots\circ F_{i, f_i-1}\circ P_{i, f_i-1}\circ F_{i, f_i}$, where each connecting path $P_{i, j}$ has length at most $3$.
\end{lemma}
\begin{proof}
We first prove (i). Let $\gamma_1$ be a real number such that $\frac{\gamma_1^2}{2}\geq \gamma$ and $2\gamma_1^3\leq \gamma$. We construct a family $\mathcal{F}^\prime$ of $4$-sets from $[V(D^\prime)]^4$ at random by including each of $\binom{n}{4}\sim n^4$ possible $4$-sets independently with probability $\gamma_1^3 n^{-3}$ (Note that some of the selected $4$-sets may not be absorbing at all). By Chernoff's inequality, since $\mathbb{E}|\mathcal{F}^\prime|=n^4\cdot\gamma_1^3 n^{-3}=\gamma_1^3n$, we have: $$\mathbb{P}(|\mathcal{F}^\prime|\geq2\gamma_1^3n)\leq\mathbb{P}(||\mathcal{F}^\prime|-\mathbb{E}|\mathcal{F}^\prime||\geq \gamma_1^3n)\leq2e^{-\frac{1}{3}\mathbb{E}|\mathcal{F}^\prime|}.$$ Thus, with probability $1-o(1)$, as $n\rightarrow \infty$:

\smallskip

\emph{$(1)$ $|\mathcal{F}^\prime|<2\gamma_1^3n\leq\gamma n$. Similarly, for every vertex pair $(u, v)$, we also have that $|\mathcal{A}_{uv}\cap\mathcal{F}^\prime|>\gamma_1^4n/3$.}

\smallskip

\noindent Next, we bound the expected number of intersecting pairs of $4$-sets in $\mathcal{F}^\prime$. The expected number is at most
\begin{equation*}
\begin{split}
n^4\times (C_4^1n^3+C_4^2n^2+C_4^3n)\times(\gamma_1^3 n^{-3})^2\leq n^4\times4\times4\times n^3\times(\gamma_1^3 n^{-3})^2=16\gamma_1^6n,
\end{split}
\end{equation*}
By Markov's inequality, with $X$ denoting the number of intersecting pairs of $4$-sets in $\mathcal{F}^\prime$ and $a=17\gamma_1^6n$, we get: $$\mathbb{P}(X\geq 17\gamma_1^6n)\leq\frac{\mathbb{E}X}{a}=\frac{16\gamma_1^6n}{17\gamma_1^6n}=16/17.$$ This implies that

\smallskip

\emph{$(2)$ with probability at least $1/17$, as $n\rightarrow\infty$, there are at most $17\gamma_1^6n$ pairs of intersecting $4$-sets in $\mathcal{F}^\prime$.}

\smallskip

\noindent Combining $(1)$ and $(2)$, we conclude that with positive probability, the family $\mathcal{F}^\prime$ satisfies both properties, which implies that there exists one such family, and, for simplicity, we define this family to be $\mathcal{F}^{\prime\prime}$. From $\mathcal{F}^{\prime\prime}$ we delete all $4$-sets that intersect other $4$-sets, as well as all $4$-sets that are not absorbers, and denote the remaining subfamily by $\mathcal{F}$. Clearly, by $(1)$ and $(2)$ again, we have: $$|\mathcal{F}|\geq2\gamma_1^3n-17\gamma_1^6n\geq\frac{3\gamma_1^3n}{2}.$$ Moreover, the family $\mathcal{F}$ consists of disjoint absorbers, and for every vertex pair $(u, v)$,
\begin{equation*}
\begin{split}
|\mathcal{A}_{uv}\cap\mathcal{F}|>\frac{\gamma_1^4n}{3}-2\cdot17\gamma_1^6n>\frac{\gamma_1^4n}{4}\geq \gamma^2n.
\end{split}
\end{equation*}
This completes the proof of (i).

\smallskip

Next we prove (ii). To establish property (ii)-$1$, we probabilistically construct the partition $\mathcal{F}=\mathcal{F}_1\cup \cdots\cup\mathcal{F}_l$ with $|\mathcal{F}_i|=f_i$ for all $i\in[l]$. 
For each absorber $A\in \mathcal{F}$, assign it uniformly at random to one of the $l$ subsets $\{\mathcal{F}_j\}_{j=1}^l$, ensuring that the cardinality condition $|\mathcal{F}_i|=f_i$ is maintained for all $i\in[l]$. Define $F_{uv}^j$ as the random variable counting the number of absorbers of $(u, v)$ in $\mathcal{F}_j$ for $j\in[l]$. From (i), $|\mathcal{A}_{uv}\cap\mathcal{F}|\geq\gamma^2n$ and $|\mathcal{F}|<\gamma n$, giving
\begin{equation*}
\begin{split}
\mathbb{E}F_{uv}^j=\frac{|\mathcal{A}_{uv}\cap\mathcal{F}|}{|\mathcal{F}|}\cdot|\mathcal{F}_j|
\geq\frac{\gamma^2n}{\gamma n}\cdot |\mathcal{F}_j|=\gamma |\mathcal{F}_j|.
\end{split}
\end{equation*}

Applying Chernoff's inequality with deviation parameter $a=\frac{1}{2}$, and noting $|\mathcal{F}_j|\geq\lambda |\mathcal{F}|\geq\lambda\gamma^2n$ (from $|\mathcal{F}|\geq|\mathcal{A}_{uv}\cap\mathcal{F}|\geq \gamma^2n$), we bound the failure probability for any $j\in[l]$:
\begin{align}\label{eee}
\mathbb{P}\left(F_{uv}^j<\frac{\mathbb{E}F_{uv}^j}{2}\right)\leq\mathbb{P}\left(|F_{uv}^j-\mathbb{E}F_{uv}^j|
>\frac{\mathbb{E}F_{uv}^j}{2}\right)
<2\exp\left({-\frac{\mathbb{E}F_{uv}^j}{12}}\right)\leq2\exp\left(-\frac{\lambda\gamma^3n}{12}\right).
\end{align}
Union bounding over all $\binom{n}{2}<n^2$ vertex pairs in $D^\prime-V(\mathcal{F})$, the total failure probability satisfies: 
\begin{equation*}
\begin{split}
\sum_{(u, v)}\mathbb{P}\left(\exists j: F_{uv}^j<\frac{\gamma|\mathcal{F}_j|}{2}\right)
<2n^2\cdot\exp\left(-\frac{\lambda\gamma^3n}{12}\right)\rightarrow 0,\ \mbox{as} \ n\rightarrow\infty.
\end{split}
\end{equation*}
Thus, with probability approaching $1$, all $\mathcal{F}_j$ contain at least $\frac{\gamma|\mathcal{F}_j|}{2}>0$ absorbers for $(u, v)$. By the probabilistic method, this guarantees the existence of a partition $\mathcal{F}=\mathcal{F}_1\cup\cdots\cup\mathcal{F}_l$, where every $\mathcal{F}_j$ contains at least one absorber for each pair $(u, v)$, establishing (ii)-$1$. 

\smallskip

Finally, we complete the proof of (ii)-$2$. Let $\mathcal{F}_1\cup \cdots\cup\mathcal{F}_l$ be a partition of $\mathcal{F}$ satisfying (ii)-$1$. For any index $i\in[l]$, we have that $|\mathcal{F}_i|=f_i$. For each $i\in[l]$, we construct $L_i$ inductively. For any $i\in[l]$, we will show by induction on $q$ that for each $q\in[f_i]$, there exists a path $S_q$ in $D^\prime$ of the form $S_1=F_{i, 1}$ and for $q\geq2$,
\begin{equation*}
\begin{split}
S_q=F_{i, 1}\circ P_{i, 1}\cdots\circ F_{i, q-1}\circ P_{i, q-1}\circ F_{i, q},
\end{split}
\end{equation*}
where each of the paths $P_{i, 1}, \ldots, P_{i, q-1}$ has the length at most $3$. Note that $L_i=S_{f_i}$.

It is obvious for the case $q=1$. Assume that the statement is true for some $q-1\in[f_i-1]$. Moreover, we suppose that the terminal (resp., the initial) of $F_{i, q-1}$ (resp., $F_{i, q}$) is $b$ (resp., $a$). 
Denote by $D_{q-1}$ the subdigraph induced by the vertex set $V_{q-1}=V(D^\prime)\setminus V((S_{q-1}-b)\cup(\mathcal{F}-a))$ in $D^\prime$. Since
\begin{equation*}
\begin{split}
|V(S_{q-1}\cup \mathcal{F})|<(4+4)\cdot|\mathcal{F}|\leq8\cdot2\gamma_1^3 n<\gamma n,
\end{split}
\end{equation*}
the conditions of Lemma \ref{s1} are satisfied. Thus,  there is a path $P_{i, q-1}$ of length at most $3$ in $D_{q-1}$ connecting $b$ to $a$. Crucially, $V(P_{i, q-1})\setminus\{a, b\}$ is disjoint from $V(\mathcal{F}\cup S_{q-1})$, and so the desired path
\begin{equation*}
\begin{split}
S_{q}=S_{q-1}\circ P_{i, q-1}\circ F_{i, q}.
\end{split}
\end{equation*}
By induction, the full path $L_i=S_{f_i}$ is constructed. The disjointness of $L_1, \ldots, L_l$ follows from the iterative removal of used vertices in $D_{q-1}$. This proves (ii)-$2$.

\smallskip

Thus, this lemma is ture.
\end{proof}
By using Fact \ref{qqqw} and Lemma \ref{absor1}, we will now present the absorption lemma.
\begin{lemma}\emph{(}Absorbing Lemma\emph{)}\label{absorbing lemma}
Let $H$ be a digraph with $k$ arcs and $\delta(H)\geq1$, and $C_0$ be a constant and parameters $\alpha, \beta, \gamma, \lambda, \varepsilon^\prime, \varepsilon$ satisfy that $0<1/C_0\ll\alpha, \beta, \gamma, \lambda\ll\varepsilon^\prime\ll\varepsilon\ll1$. Suppose $\mathcal{N}=\{n_1, \ldots, n_k\}$ is an integer set with $n_1\geq\cdots\geq n_k\geq4$ and $\sum_{n_i<\alpha n}n_i\leq\beta n$. Let $l\in[k]$ be the largest subscript such that $n_l>\alpha n$. If $D$ is a digraph of order $n\geq C_0k$ with $\delta^0(D)\geq n/2+k$, then there is an $H$-linked subdigraph $H^\prime\subseteq D$ with $|V(H^\prime)|\leq\gamma n$ such that

\noindent $(i)$ the lengths of the subdivided paths $P_1, \ldots, P_l, P_{l+1}, \ldots, P_k$ in $H^\prime$ are $n_1^\prime, \ldots, n_k^\prime$, where $n_i^\prime\leq n_i-6$ for $i\in[l]$, and $n_i^\prime=n_i$ for $i\in\{l+1, \ldots, k\}$, and

\noindent $(ii)$ for any vertex pair $(u, v)$, $P_i$ $(\mbox{for}\ i\in[l])$ contains at least one absorber for $(u, v)$.
\end{lemma}
\begin{proof}
Let $V(H)=\{v_1, \ldots, v_{|V(H)|}\}$. 
We relabel the vertices in $f(V(H))$ as $f(V(H))=\cup_{i=1}^k \{v_i, v_{i}^\prime\}$ such that, in the desired Hamiltonian $H$-linked subdigraph, the length of the path from $v_i$ to $v_{i}^\prime$ is $n_i$. 
Note that $l$ is an absolute constant independent of $n$ since $l\leq \frac{n}{\alpha n}=\frac{1}{ \alpha}$. Let $D^\prime=D-f(V(H))$, and clearly $\delta^0(D^\prime)\geq n/2-k$ due to $|f(V(H))|\leq 2k$.

\smallskip

By Lemma \ref{absor1}-(i), there exists a family $\mathcal{F}$ of at most $\gamma n$ disjoint absorbing $4$-paths in $D^\prime$ such that for every vertex pair $(u, v)$, we have $|\mathcal{A}_{uv}\cap\mathcal{F}|\geq\gamma^2n$. Additionally, Lemma \ref{absor1}-(ii) shows that $\mathcal{F}$ can be partitioned into $l$ distinct subfamilies, say $\mathcal{F}_1, \ldots, \mathcal{F}_l$, such that 
each subfamily contains at least an absorber for any vertex pair $(u, v)$ in $V(D^\prime-\mathcal{F})$. Importantly, in this process, we require the number of absorbers in each subfamily to be at least $\lambda |\mathcal{F}|$. 
Also, Lemma \ref{absor1}-(ii) ensures that we can connect all $4$-paths in $\mathcal{F}_i$ into a path $L_i$ of length at most $n_i-6$ for all $i\in[l]$, and these $l$ paths $L_1, \ldots, L_l$ are disjoint with $\sum_{i=1}^l|V(L_i)|\leq\gamma n$.

\smallskip

On the one hand, in the desired $H$-linked subdigraph $H^\prime$, for all $(v_j, v_j^\prime)$-paths with $j\in\{l+1, \ldots, k\}$, we can greedily construct these paths. Specifically, for any $j\in\{l+1, \ldots, k\}$, we can choose a vertex $u_3\in N^+_{D^\prime}(v_{j})$ and $u_{i+1}\in N^+_{D^\prime}(u_i)\setminus \{u_3, \ldots, u_{i}\}$ for all $i\in\{3, \ldots, n_{j}-3\}$, and there an arc from $N^+_{D^\prime}(u_{n_{j}-2})\setminus \{v_1, u_3, \ldots, u_{n_{j}-2}, v_2\}$ to $N^-_{D^\prime}(v_{j}^\prime)\setminus \{v_1, u_3, \ldots, u_{n_{j}-2}\}$. This is possible because $D$ does not satisfy the extremal condition, and $n_{l+1}+\cdots+n_k\leq\beta n\ll\varepsilon^\prime n$, and by the lower bound of $\delta^0(D^\prime)$, the cardinalities of these two vertex sets are greater than $(1/2-\varepsilon)n$. By repeating this process for all any $j\in\{l+1, \ldots, k\}$, we obtain all internally disjoint paths $P_{l+1}, P_{l+2}, \ldots, P_k$ of length $n_{l+1}, n_{l+2}, \ldots, n_k$, respectively, as required in the desired $H$-linked subdigraph $H^\prime$.

\smallskip

On the other hand, for each path of length $n_i> \alpha n$ for every $i\in[l]$, we will connect its end-vertices to the corresponding absorbing path $L_i$. Without loss of generality, write $L_i=a_i\cdots b_i$ for each $i\in[l]$. In the remaining digraph $D^{\prime\prime}=D^\prime-\bigcup_{i=1}^l L_i-\bigcup_{j=l+1}^kP_j$, for any $i\in[l]$, we define $R_1=N^+_{D^{\prime\prime}}(v_{i})$ and $R_2=N^-_{D^{\prime\prime}}(a_i)$. By the lower bound of $\delta^0(D)$, we have that $|R_1|, |R_2|\geq n/2-k-\beta n-\gamma n-4l\geq(1/2-\varepsilon)n$, which implies the existence of an arc $xx^\prime$ from $R_1$ to $R_2$ since $D^{\prime\prime}\subseteq D$ does not meet the extremal condition. Similarly there is an arc $yy^\prime$ from $N^+_{D^{\prime\prime}}(b_i)$ to $N^-_{D^{\prime\prime}}(v_{i}^\prime)$. This yields a path $v_{i}xx^\prime L_lyy^\prime v_{i}^\prime$ of length at most $n_i$. Repeating this process for all $i\in[l]$, we can construct all desired paths of length at most $n_{1}, n_{2}, \ldots, n_l$, respectively.

\smallskip

Let $H^\prime$ be the subdigraph obtained from the union of these paths of lengths at most $n_1, \ldots, n_k$. Clearly, $H^\prime$ is $H$-linked. By Lemma \ref{absor1}-(i), it follows directly that the property (ii) of this lemma holds. 
Hence, this completes the proof of this lemma.
\end{proof}
\subsubsection{Path-Covering}
Recall that, as shown in Lemma \ref{absorbing lemma}, we can obtain an $H$-linked subdigraph $H^{\prime}$. 
In the following lemma, we will demonstrate that $V(D-H^{\prime})$ can be covered  by a Hamiltonian path. Before presenting the proof, we introduce some definitions and a result of K\"{u}hn, Osthus and Treglown.
\begin{definition} \emph{(Robust $(\nu, \tau)$-outexpander)
Let $\nu$ and $\tau$ be real numbers with $0<\nu\leq\tau<1$. Suppose $D$ is a digraph, and the vertex subset $S\subseteq V(D)$. The \emph{$\nu$-robust out-neighbourhood $RN^+_{\nu, D}(S)$ of $S$} is defined as the set of all vertices $x$ in $D$ that have at least $\nu|V(D)|$ in-neighbourhoods in $S$. Moreover, $D$ is called a \emph{robust $(\nu, \tau)$-outexpander} if
$|RN^+_{\nu, D}(S)|\geq|S|+\nu|V(D)|$ for all $S\subseteq V(D)$ with $\tau|V(D)|<|S|<(1-\tau)|V(D)|$}.
\end{definition}
The proof of Lemma \ref{Path-Cover Lemma} relies on a result of K\"{u}hn, Osthus and Treglown \cite{Kuhn1}, which establishes the existence of a Hamiltonian cycle in a digraph with a small lower bound on semi-degree and a certain expansion property.
\begin{theorem}\cite{Kuhn1}\label{KK}
Let $n_0$ be a positive integer, and let $\nu, \tau$ and $\xi$ be positive constants such that $1/n_0\ll\nu\leq\tau\ll\xi< 1$. If $D$ is a digraph on $n\geq n_0$ vertices with $\delta^0(D)\geq\xi n$ and is a robust $(\nu, \tau)$-outexpander, then $D$ contains a Hamiltonian cycle.
\end{theorem}
Theorem \ref{KK} and its undirected version have been utilized as a black box in various papers, including \cite{Ferber,Kuhn00,Kuhn0,Kuhn2,Knox,Staden}. It's worth noting that Theorem \ref{KK} originally relies on Regularity Lemma. However, in $2012$, K\"{u}hn and Osthus \cite{K2} gave a brief proof of Theorem \ref{KK}, and in $2018$, Lo and Patel \cite{Allan} provided a proof of Theorem \ref{KK} by applying ``sparse'' robust expanders instead of Regularity Lemma.

\smallskip

Now, we present the statement of the Path-Covering lemma. 
\begin{lemma}\emph{(}Path-Covering Lemma\emph{)}\label{Path-Cover Lemma} 
Let $H$ be a digraph with $k$ arcs and $\delta(H)\geq1$, and let $D$ be a stable digraph of order $n\geq C_0k$ with $\delta^0(D)\geq n/2+k$. Suppose the parameters $\gamma$ and $\varepsilon^\prime$ satisfy $0<1/C_0\ll\gamma\ll\varepsilon^\prime\ll1$. If $H^\prime\subseteq D$ is an $H$-linked subdigraph with $|V(H^\prime)|\leq\gamma n$, then
the digraph $D-H^\prime$ contains a Hamiltonian path.
\end{lemma}
\begin{proof}
Let $\eta$ be a real number with $\gamma\ll\eta<\varepsilon^\prime/3$, and define $D^\prime=D-H^\prime$. Clearly, $$\delta^0(D^\prime)\geq(n/2+k)-\gamma n\geq (1/2-\eta)n.$$ 
Let $\nu$ and $\tau$ be positive constants such that $\nu\ll\tau/2\ll\gamma$ and $\nu\leq(\varepsilon^\prime)^2$.

We now demonstrate that $D^\prime$ is a robust $(\nu, \tau)$-outexpander by considering three cases for any vertex set $S\subseteq V(D^\prime)$. Firstly, if $(1/2+\varepsilon^\prime/2)n<|S|<(1-\tau)n$, then we can deduce that $RN_{\nu, D^\prime}^+(S)=V(D^\prime)$. This is because the lower of $\delta^0(D^\prime)$ guarantees that for any vertex $x\in V(D^\prime)$,
\begin{equation*}
\begin{split}
d^-_{S}(x)\geq \delta^0(D^\prime)-|\overline{S}|>(1/2-\eta)n-(1/2-\varepsilon^\prime/2)n=(\varepsilon^\prime/2-\eta)n\geq\nu n,
\end{split}
\end{equation*}
where the last inequality holds because $\nu\ll\eta<\varepsilon^\prime/3$.

\smallskip

Secondly, if $\tau n<|S|<(1/2-\varepsilon^\prime/2)n$, then we have that $|RN^+_{\nu, D^\prime}(S)|\geq|S|+\nu n$. Actually, using the lower bound of $\delta^0(D^\prime)$ once again, we obtain that
\begin{equation*}
\begin{split}
|S|\cdot\delta^0(D^\prime)\leq\sum_{y\in S}d^+(y)<|RN^+_{\nu, D^\prime}(S)|\cdot|S|+(n-|RN^+_{\nu, D^\prime}(S)|)\cdot\nu n.
\end{split}
\end{equation*}
Since $|S|<(1/2-\varepsilon^\prime/2)n$, we that
\begin{equation*}
\begin{split}
|S|\cdot(1/2-\eta)n-\nu n^2<|RN^+_{\nu, D^\prime}(S)|\cdot(|S|-\nu n)<|RN^+_{\nu, D^\prime}(S)|((1/2-\varepsilon^\prime/2)-\nu)n.
\end{split}
\end{equation*}
Rearranging this inequality, and using the fact that $\eta<\varepsilon^\prime/3$ and $\tau/2\ll\gamma\ll\varepsilon^\prime$ to further simplify, we get:
\begin{equation*}
\begin{split}
|RN^+_{\nu, D^\prime}(S)|>\frac{|S|\cdot(1/2-\eta)n-\nu n^2}{(1/2-\varepsilon^\prime/2)n-\nu n}
&=|S|+\frac{|S|(-\eta+\varepsilon^\prime/2+\nu)-\nu n}{(1/2-\varepsilon^\prime/2)-\nu }\\
&>|S|+\frac{|S|(\varepsilon^\prime/6+\nu)-\nu n}{1/2-\varepsilon^\prime/2-\nu}\geq|S|+\nu n,
\end{split}
\end{equation*}
where the last inequality follows from $\tau n<|S|$ and $\nu\ll\tau/2\ll\varepsilon^\prime$, which imply:
$$|S|(\varepsilon^\prime/6+\nu)-\nu n>\tau n\cdot(\varepsilon^\prime/6+\nu)-\nu n=\tau\varepsilon^\prime n/6-\nu n+\tau\nu n>\nu n/2-\varepsilon^\prime \nu n-\nu^2 n.$$

\smallskip

Finally we consider the case when $(1/2-\varepsilon^\prime/2)n\leq|S|\leq(1/2+\varepsilon^\prime/2)n$. To obtain a contradiction, assume $|RN^+_{\nu, D^\prime}(S)|<|S|+\tau n\leq(1/2+\varepsilon^\prime/2)n+\tau n$. Then we have that $|V(D^\prime)\setminus RN^+_{\nu, D^\prime}(S)|\geq(1/2-\varepsilon^\prime)n$. Sine $D^\prime$ is stable (as $D$ is stable and $D^\prime\subset D$), there are at least $(\varepsilon^\prime n)^2$ arcs from $S$ to $V(D^\prime)\setminus RN^+_{\nu, D^\prime}(S)$. This is because $|S|, |V(D^\prime)\setminus RN^+_{\nu, D^\prime}(S)|\geq(1/2-\varepsilon^\prime)n$ and $(U_1, U_2)_{\textbf{EC}}=(S, V(D^\prime)\setminus RN^+_{\nu, D^\prime}(S))$. On the other hand, by the definition of $RN^+_{\nu, D^\prime}(S)$, each vertex $z\in V(D^\prime)\setminus RN^+_{\nu, D^\prime}(S)$ has fewer than $\nu n$ in-neighbourhoods in $S$, which suggests that
\begin{equation*}
\begin{split}
e^+(S, V(D^\prime)\setminus RN^+_{\nu, D^\prime}(S))<|V(D^\prime)\setminus RN^+_{\nu, D^\prime}(S)|\cdot \nu n\leq \nu n^2.
\end{split}
\end{equation*}
This implies that $(\varepsilon^\prime)^2<\nu$, a contradiction. 
Therefore, $|RN^+_{\nu, D^\prime}(S)|\geq|S|+\nu n$ as desired.

Hence, we have shown that $D^\prime$ is a robust $(\nu, \tau)$-outexpander. By Theorem \ref{KK}, we conclude that $D^\prime$ contains a Hamiltonian cycle, which confirms this lemma.
\end{proof}
\subsubsection{Completion of Theorem \ref{song2}}
Recall that $H$ is a digraph with $k$ arcs and $\delta(H)\geq1$. Let $C_0$ be an integer, and let $\alpha_0, \beta_0 \in (0,1)$ be two real numbers. Fix $\alpha\in (0,\alpha_0]$ and $\beta\in (0,\beta_0]$. Suppose $D$ is a digraph on $n\geq C_0k$ vertices with $\delta^0(D)\geq n/2+k$. The parameters $\alpha, \beta, \varepsilon$, $\varepsilon_1$ and $\varepsilon^\prime$ satisfy $0<1/C_0\ll\alpha, \beta\ll\varepsilon^\prime\ll\varepsilon_1\ll\varepsilon\ll1$.

Let $\mathcal{N}=\{n_1, \ldots, n_k\}$ be a set of integers with $n_1\geq n_2\geq\cdots\geq n_k\geq4$ and
 $\sum_{n_i<\alpha n}n_i\leq\beta n$. Let $l\in[k]$ be the largest subscript such that $n_l>\alpha n$. By the Absorbing Lemma (Lemma \ref{absorbing lemma}), we obtain an $H$-linked subdigraph in $D$, called as $H^\prime$, with $|V(H^\prime)|\leq\gamma n$ and satisfying that

\noindent\emph{$(i)$ $H^\prime$ contains all paths $P_{l+1}, \ldots, P_k$ of lengths $n_{l+1}, \ldots, n_k$, respectively, and }

\noindent\emph{$(ii)$ in $H^\prime$, there exist $l$ `long' paths, defined as $P_1, \ldots, P_l$, where each $P_i$ has length less than $n_i$ for $i\in[l]$. Additionally, each `long' path $P_i$ contains an absorber for any vertex pair $(u, v)$ of $D-H^\prime$.}

\smallskip

Next, we apply the Path-Covering Lemma (Lemma \ref{Path-Cover Lemma}) to get a Hamiltonian path $P$ of the digraph $D-H^\prime$. Then we partition this path $P$ into $l$ disjoint paths of appropriate lengths, denoted as $Q_1, \ldots, Q_l$. For each $i\in[l]$, the path $Q_i$ is of the form $Q_i=c_i\cdots d_i$ and its the number of vertices satisfies $|V(Q_i)|=n_i-|V(P_i)|+1$.
By Lemma \ref{Path-Cover Lemma}-(ii), the vertex pair $(c_i, d_i)$ has an absorber in the path $P_i$. Then the path $Q_i$ can be absorbed into $P_i$. Define $P^\prime_i=Q_i\cup P_i$ for each $i\in[l]$. Then, $P^\prime_i$ is a path of length exactly $n_i$.

\smallskip

At this stage, we have constructed a collection of paths $P_1^\prime, \ldots, P_l^\prime$ of lengths $n_1, \ldots, n_l$, respectively, and in $H^\prime$, there are the paths $P_{l+1}, \ldots, P_{k}$ of lengths $n_{l+1}, \ldots, n_k$, respectively. Together, these form a Hamiltonian $H$-linked subdigraph of $D$. This completes the proof of Theorem \ref{song2} for the case when $D$ is stable.

\subsection{Extremal case}
Let $H$ be a digraph with $k$ arcs and $\delta(H)\geq1$. In this subsection, we always assume:\\
(i) $C_0$ is a positive integer, and $\alpha_0, \beta_0 \in (0,1)$ are two real, as defined in Theorem \ref{song2},\\
(ii) $D$ is a digraph on $n\geq C_0k$ vertices with $\delta^0(D)\geq n/2+k$ and $D$ is not stable, and\\
(iii) $\alpha\in (0,\alpha_0]$ and $\beta\in (0,\beta_0]$, and parameters $\alpha, \beta, \varepsilon, \varepsilon_1$ and $\varepsilon^\prime$ are chosen such that $$0<1/C_0\ll\alpha, \beta \ll\varepsilon^\prime\ll\varepsilon_1\ll\varepsilon\ll1.$$
Let $\mathcal{N}=\{n_1, \ldots, n_k\}$ be an integer set where each $n_i\geq4$ for $i\in[k]$ and $\sum_{n_i<\alpha n}n_i\leq\beta n$. Clearly, $k\leq\frac{\alpha n}{5}+\frac{1}{\beta}$. We relabel the vertices in $f(V(H))$ as $f(V(H))=\cup_{i=1}^k \{v_i, v_{i}^\prime\}$ such that, in the desired Hamiltonian $H$-linked subdigraph, the length of the path from $v_i$ to $v_{i}^\prime$ is $n_i$.

\smallskip

We first define the \emph{strong neighbourhood} of a vertex $x$ in $D$ to be $SN(x)=\{y: xy, yx\in A(D)\}$, and the \emph{strong semi-degree} of $x$ in $D$, defined  $s(x)$, as the cardinality of $SN(x)$, i.e.,  $s(x)=|SN(x)|$. Also, for a vertex subset $U$ of $D$, let $s_U(x)=|SN(x)\cap U|$. Additionally, we introduce the following definitions, which will be frequently used in this section:
\begin{definition}\label{definition3.2}
\emph{Let $U_1$ and $U_2$ be two disjoint vertex subsets in $V(D)$, and let $u\in U_1$ (resp., $v\in U_2$). We define exceptional vertices of \emph{Types I$_1$-I$_4$}  with respect to $U_1$ and $U_2$, respectively, as follows. For each $u$ ($v$, respectively), we say that $u$ ($v$, respectively) is of\\
(i) \emph{Type I$_1$}, if, for some $\sigma\in\{+, -\}$,
$d_{U_1}^{\sigma}(u)\leq(1-\sqrt{10\varepsilon})|U_1|$ ($d_{U_2}^{\sigma}(v)\leq(1-\sqrt{10\varepsilon})|U_2|$, $~~~~~~$respectively).\\
(ii) \emph{Type I$_2$}, if, for some $\sigma\in\{+, -\}$,
$d_{U_1}^\sigma(u)\leq\varepsilon^{1/3}|U_1|$ ($d_{U_2}^\sigma(v)\leq\varepsilon^{1/3}|U_2|$, respectively).\\
(iii) \emph{Type I$_3$}, if $s_{U_{2}}(u)\leq(1-\sqrt{10\varepsilon})|U_{2}|$ ($s_{U_{1}}(v)\leq(1-\sqrt{10\varepsilon})|U_{1}|$, respectively).\\
(iv) \emph{Type I$_4$}, if $s_{U_{2}}(u)\leq\varepsilon^{1/3}|U_{2}|$ ($s_{U_{1}}(v)\leq\varepsilon^{1/3}|U_{1}|$, respectively).
}
\end{definition}
For each $i\in[4]$, we also use $E_i$ to represent the set of vertices of Type I$_i$ in $D$. It is clear that for every $i\in\{1, 3\}$, we have $E_{i+1}\subseteq E_{i}$.
\begin{definition}\label{definition3.3}
\emph{Let $U_1$ and $U_2$ be two disjoint vertex sets of $V(D)\setminus f(V(H))$. For any $j\in[2]$, define: \\
(i) $V_j$ as the set of vertex pairs $(v_{i}, v_{i}^\prime)$ such that
$$|N^+(v_{i})\cap U_j|\geq 4k\ \mbox{and}\ |N^-(v_{i}^\prime)\cap U_j|\geq 4k.$$\\
(ii) $V_{j+2}$ as the set of vertex pairs $(v_{i}, v_{i}^\prime)$ such that
$$|N^+(v_{i})\cap U_j|\geq 4k\ \mbox{and}\ |N^-(v_{i}^\prime)\cap U_{3-j}|\geq 4k.$$}
\end{definition}
We now present the following proposition and lemma, which are simple yet interesting and will be repeatedly used in the extremal cases.
\begin{proposition}\label{prop}
Suppose that $C$ is a positive integer and $\eta$ is any real number satisfying $1/C\ll\eta\ll1$. Consider an integer partition $a=a_1+ \cdots+a_k$ with $a\geq Ck$.
Let $T$ be a digraph with vertex set $V(T)=A\cup B$, where $A\cap B=\emptyset$ and $|A|=|B|=a$. Suppose that for any $\sigma\in\{+, -\}$, the following holds: for any vertex $u\in A$ and any $v\in B$, $d^\sigma_B(u)\geq(1-\eta)a$ and $d^\sigma_A(v)\geq(1-\eta)a$, respectively. Then for any vertex set $U\subseteq V(T)$ such that $U\cap A=\{x_1^0, \ldots, x_k^0\}$ and $U\cap B=\{y_1^0, \ldots, y_k^0\}$, the digraph $T$ contains $k$ disjoint paths $P_1, \ldots, P_k$ satisfying the following for each $j\in[k]$.\\
$(i)$ The initial and the terminal of $P_j$ is $x_j^0$ and $y_j^0$, respectively.\\
$(ii)$ $|V(P_j)\cap A|=|V(P_j)\cap B|=a_j$.
\end{proposition}
\begin{proof}
For convenience, let $r_j=a_j-1$ for any $j\in[k]$.
For each $j\in[k]$, we choose $r_j+1$ vertices $x_j^0, x_j^1, \ldots, x_j^{r_j}$ from $A$ with the last vertex $x_j^{r_j} \in N^-_A(y_j^0)$ such that all selected vertices are distinct and their union covers $A$. We construct an auxiliary bipartite graph $Q=(\tilde{A}, B^\prime)$ such that $\tilde{A}=\bigcup_{j=1}^k\{(x_j^0, x_j^1), (x_j^1, x_j^2), \ldots, (x_j^{r_j-1}, x_j^{r_j})\}$ and $B^\prime=B\setminus U$ where each `vertex' $(x_j^i, x_j^{i+1})$ in $\tilde{A}$ connects with all the vertices in $N_{T}^+(x_j^i)\cap N_{T}^-(x_j^{i+1})$.
Obviously, any perfect matching in $Q$ that saturates $\tilde{A}$ corresponds to an embedding of $P_1, \ldots, P_k$ in $T$ as required. We claim that such perfect matching exists. In fact,  $|\tilde{A}|=|B^{\prime}|=a-k$ and $d_Q(z)\geq2(1-\eta)|B^\prime|-|B^\prime|\geq(1-2\eta)|B^\prime|$ for $z\in \tilde{A}$. Additionally, we deduce that $d_Q(u)\geq(1-2\eta)|\tilde{A}|$ for any vertex $u\in B^\prime$. Therefore, the degrees of the vertices in $Q$ are all at least $(1-2\eta)(a-k)$. Then by the K\"{o}nig-Hall's theorem, we conclude that $Q$ has a perfect matching. This matching corresponds to the desired paths $P_1, \ldots, P_k$ in $T$, completing the proof of this proposition.
\end{proof}
For any two vertex subsets $X$ and $Y$ of $V(D)$ and a parameter $0<\varepsilon\ll1$, we say $X$ is \emph{$\varepsilon$-approximately equal to} $Y$ if $|X|=|Y|\pm \varepsilon n$.
We now define the following extremal case, which occurs when $D$ satisfies the extremal condition (\textbf{EC}).


%
%
%
\begin{definition}\label{ec}
\emph{(\textbf{Extremal Case $\mathbf{1}$ (EC1) with parameter $\varepsilon$})}
\emph{The vertex set $V(D)$ can be partitioned into four disjoint vertex sets $W_1, W_2, W_3$ and $W_4$ 
such that $|W_1|=|W_3|\pm \varepsilon n$, and $|W_2|=|W_4|\pm \varepsilon n$. Furthermore, the following conditions hold. }
\begin{itemize}
\item[$(A)$]
\emph{\textbf{Almost one-way completeness:} For each $i\in[4]$, $e^+(W_i, W_{i+1})\geq|W_i|\cdot|W_{i+1}|-\varepsilon n^2$, where $W_5=W_1$. In particular, we also say that $D[W_i\cup W_{i+1}]$ is $\varepsilon$-\emph{almost one-way complete} for each $i\in [4]$ and $W_5=W_1$.}
\item[$(B)$]
\emph{\textbf{Almost completeness:} For each $i\in\{1, 3\}$, $e(W_i)\geq|W_i|^2-\varepsilon n^2$. In this case, we also say that $D[W_i]$ is $\varepsilon$-\emph{almost complete}.}
\item[$(C)$]
\emph{\textbf{Almost complete bipartite:} For each $i\in\{2, 4\}$, $e^+(W_i, W_{i+2})\geq|W_i|\cdot|W_{i+2}|-\varepsilon n^2$, where $W_6=W_2$. In this case, we say that $(W_2, W_4)$ is a $\varepsilon$-\emph{almost complete bipartite} pair.}
\end{itemize}
\end{definition}

Based on the extremal condition (\textbf{EC}) and the definition of \textbf{EC1}, we can use traditional structural analysis methods to effectively establish the following result.
\begin{lemma}\label{claim1}
Suppose constants $k, C_0>0$, and parameters $\varepsilon^\prime, \varepsilon$ satisfy $1/C_0\ll\varepsilon^\prime \ll\varepsilon\ll1$. If $D$ is a digraph of order $n\geq C_0k$ with $\delta^0(D)\geq n/2+k$, and satisfies the extremal condition \emph{(}\textbf{EC}\emph{)} with parameter $\varepsilon'$, then $D$ belongs to \textbf{EC1} with parameter $\varepsilon$.
\end{lemma}
\begin{proof}
Since $D$ satisfies \textbf{EC}, there exist two (not necessarily disjoint) vertex sets $U_1$ and $U_2$ with $|U_i|\geq(1/2-\varepsilon^\prime)n$ for every $i\in[2]$, and $e^+(U_1, U_2)\leq(\varepsilon^\prime n)^2$. For convenience, let $U_0:=U_1\cap U_2$. We consider the case by case based on the cardinality of $U_0$.

\smallskip

Choose a new parameter $\varepsilon_1$ such that $1/C_0\ll\varepsilon^\prime\ll \varepsilon_1\ll \varepsilon\ll1$.

\smallskip

\noindent\textbf{Case 1. $|U_0|\leq\varepsilon_1 n$.} 

We first define
$W_1=U_1\setminus U_0$, $W_3=U_2\setminus U_0$, and $W_2=W_4=\emptyset$. Clearly $W_1$ and $W_3$ are disjoint, and $e^+(W_1, W_3)\leq e^+(U_1, U_2)\leq(\varepsilon^\prime n)^2$. Additionally, for every $i\in\{1, 3\}$, since $\varepsilon^\prime\ll\varepsilon_1\ll\varepsilon$, we have: $$|W_1|=|U_1\setminus(U_1\cap U_2)|\geq(1/2-\varepsilon^\prime-\varepsilon_1)n\geq(1/2-\varepsilon/2)n,$$ and similarly, $|W_3|\geq(1/2-\varepsilon/2)n$.
Further, together with $\delta^0(D)\geq n/2+k$, $|W_1|\leq(1/2+\varepsilon/2)n$, $|V(D)\setminus(W_1\cup W_3)|\leq 2\varepsilon n$, $e^+(W_1, \overline{W_1})=e^+(W_1, W_3)+e^+(W_1, V(D)\setminus(W_1\cup W_3))$, $1/C_0\ll\varepsilon^\prime\ll\varepsilon_1\ll\varepsilon$ and $k\leq n/C_0\leq \varepsilon^\prime n$, we can deduce that
\begin{align}\label{eqq}
e(W_1)\geq\sum_{u\in W_1}d^{+}(u)-e^+(W_1, \overline{W_1})&\geq|W_1|\cdot(n/2+k)-(\varepsilon^\prime n)^2-(1/2+\varepsilon/2)n\cdot2\varepsilon n\nonumber\\
&\geq|W_1|^2-\varepsilon n^2.
\end{align}
Following the same calculation as in $(\ref{eqq})$, we can sum the in-degrees of vertices in $W_3$ to obtain that $$e(W_3)\geq|W_3|^2-\varepsilon n^2.$$ Therefore, $D[W_1]$ and $D[W_3]$ are $\varepsilon$-almost complete. It is easy to see that $|W_1|=|W_3|\pm \varepsilon n$, and
thus by splitting $V(D)\backslash (W_1\cup W_3)$ into arbitrary parts $W_2$ and $W_4$, we obtain that $D$ satisfies $(A)$-$(C)$, and thus the conclusion holds (see Figure \ref{fig1} $(a)$).

\begin{figure}[h]
\centering
\scriptsize
\begin{tabular}{ccc}\label{7}
\includegraphics[width=4.5cm]{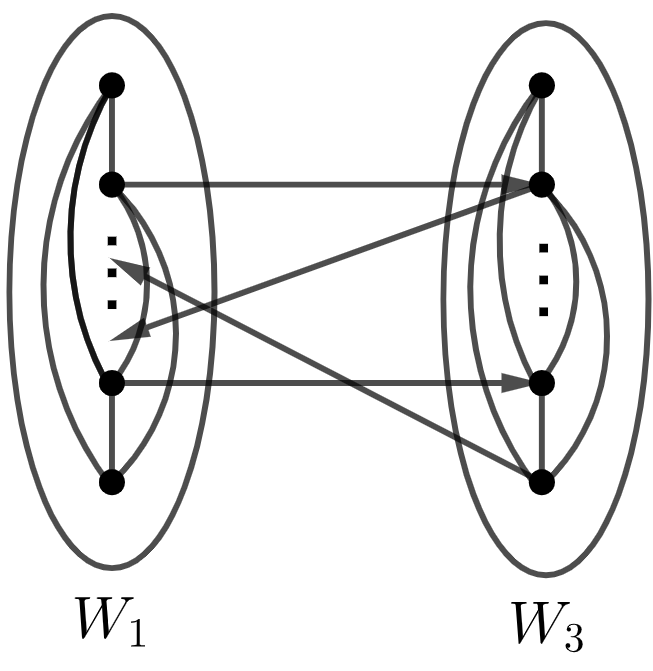}&\includegraphics[width=4.6cm]{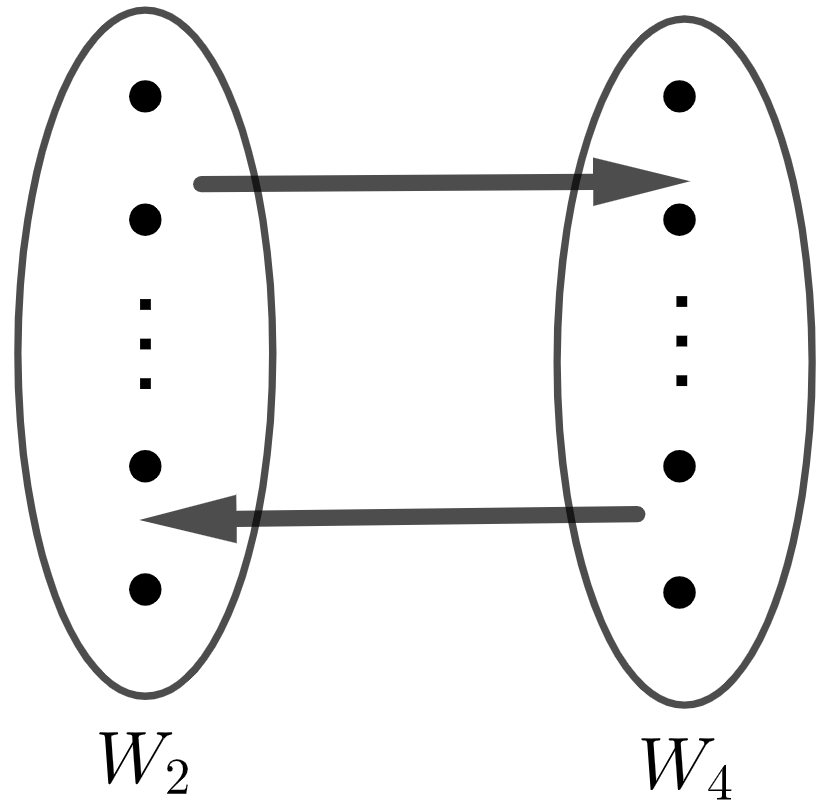}&\includegraphics[width=4.9cm]{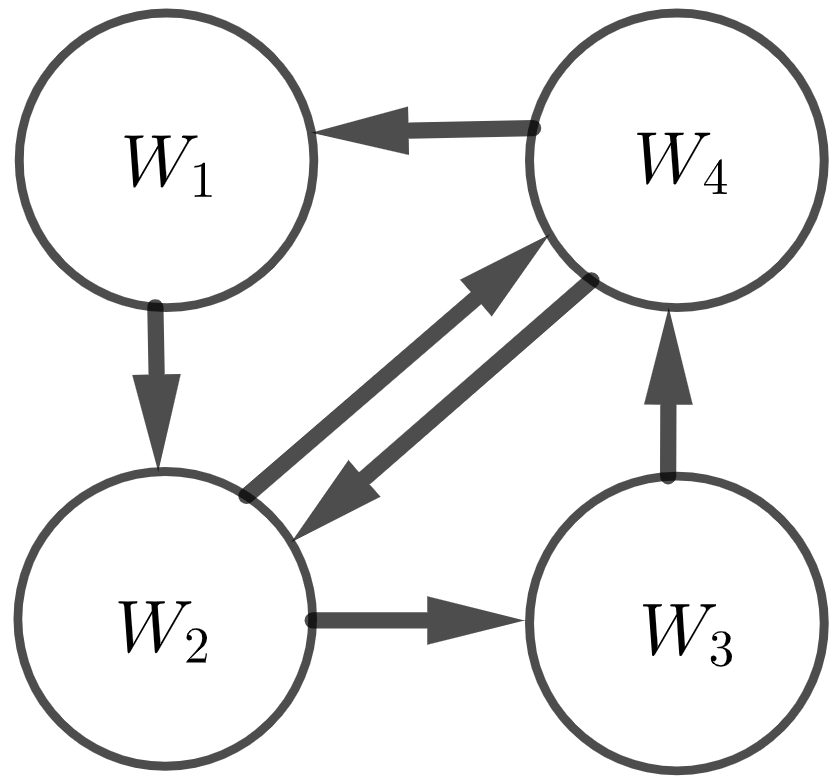}\\
(a) The case when $|U_0|\leq\varepsilon_1 n$. & (b) The case of $|U_0|\geq(1/2-\varepsilon_1)n$. & (c) The case of $\varepsilon_1n<|U_0|<(1/2-\varepsilon_1)n$.
\end{tabular}
\caption{In this figure, an edge without direction between two vertices indicates a $2$-cycle, and a black arrow pointing between two vertex sets indicates that the reduced digraph formed by them is $\varepsilon$-almost one-way complete.}
\label{fig1}
\vspace{-0.5em}
\end{figure}

\smallskip

\noindent \textbf{Case 2. $|U_0|\geq(1/2-\varepsilon_1)n$.}

Without loss of generality, we assume that $|U_0|=(1/2-\varepsilon_1)n$.
 Let $W_2=U_0$, $W_4=V(D)\setminus U_0$, and $W_1=W_3=\emptyset$.
It is evident that $e(W_2)\leq e(U_0)\leq e^+(U_1, U_2)\leq(\varepsilon^\prime n)^2$. Combining with $\delta^0(D)\geq n/2+k$ and $|W_2|=(1/2-\varepsilon_1)n$, we conclude that
\begin{align}\label{eqq1}
e^+(W_2, W_4)&\geq|W_2|\cdot(n/2+k)-(\varepsilon^\prime n)^2\nonumber\\
&=|W_2|\cdot(1/2+\varepsilon_1)n+k|W_2|-|W_2|\cdot\varepsilon_1n-(\varepsilon^\prime n)^2\nonumber\\
&\geq|W_2|\cdot|W_4|+k|W_2|-(\varepsilon_1+(\varepsilon^\prime)^2)n^2\nonumber\\
&\geq|W_2|\cdot|W_4|-\varepsilon n^2
\end{align}
by $k\leq n/C_0\leq \varepsilon^\prime n$ and $\varepsilon^\prime\ll\varepsilon_1\ll\varepsilon\ll1$.
Similar to $(\ref{eqq1})$, by calculating the sum of in-degrees of vertices in $W_2$, we can also obtain that $$e^+(W_4, W_2)\geq|W_2|\cdot|W_4|-\varepsilon n^2.$$ It follows that $|W_2|=|W_4|\pm \varepsilon n$, and so $(W_2, W_4)$ is $\varepsilon$-almost complete bipartite. Hence, according to the definition of \textbf{EC1}, we conclude $D$ belongs to \textbf{EC1} (see Figure \ref{fig1} $(b)$).

\smallskip

\noindent\textbf{Case 3. $\varepsilon_1n<|U_0|<(1/2-\varepsilon_1)n$.}

Let $W_1=U_1\setminus U_0$, $W_2=V(D)\setminus(U_1\cup U_2)$, $W_3=U_2\setminus U_0$ and $W_4=U_0$. We first estimate the cardinalities of $W_1$, $W_2$, $W_3$ and $W_4$. The following conclusion holds.
\begin{claim}\label{ccclaim}
For each $j\in\{1, 3\}$, $(1/2-\varepsilon^\prime/2)n+k\leq|W_j|+|W_2|\leq(1/2+\varepsilon^\prime)n$.
\end{claim}
\begin{proof}
On the one hand, since $e(W_4)+
e^+(W_4, W_3)=e(U_0)+
e^+(U_0, U_2\setminus U_0)\leq e^+(U_1, U_2)\leq(\varepsilon^\prime n)^2$, by calculating the out-degrees of vertices in $W_4$, we have that
\begin{align*}
(n/2+k)\cdot|W_4|\leq \sum_{w\in W_4}d^{+}(w)&=e^+(W_4, W_1)+e(W_4)+
e^+(W_4, W_3)+e^+(W_4, W_2) \nonumber \\
&\leq e^+(W_4, W_1)+(\varepsilon^\prime n)^2+e^+(W_4, W_2) \nonumber \\
&\leq|W_4|\cdot|W_1|+(\varepsilon^\prime n)^2+|W_4|\cdot|W_2|.
\end{align*}
Since $|W_4|=|U_0|>\varepsilon_1 n$ and $\varepsilon^\prime\ll \varepsilon_1$, this implies that $$(1/2-\varepsilon^\prime/2)n+k\leq|W_1|+|W_2|.$$ Similarly, by calculating the in-degrees of vertices of $W_4$, we also obtain that
\begin{align*}
(n/2+k)\cdot|W_4|\leq \sum_{w\in W_4}d^{-}(w)&=e^-(W_4, W_1)+e(W_4)+
e^-(W_4, W_3)+e^-(W_4, W_2) \nonumber \\
&\leq (\varepsilon^\prime n)^2+e^-(W_4, W_3)+e^-(W_4, W_2) \nonumber \\
&\leq (\varepsilon^\prime n)^2+|W_4|\cdot|W_3|+|W_4|\cdot|W_2|.
\end{align*}
Thus we get that $$(1/2-\varepsilon^\prime/2)n+k\leq|W_2|+|W_3|.$$

On the other hand, due to $|U_i|\geq(1/2-\varepsilon^\prime)n$ for each $i\in[2]$, it can be deduced that $$|W_1|+|W_2|=|V(D)\setminus U_2|\leq(1/2+\varepsilon^\prime)n\ \mbox{and}\ |W_2|+|W_3|=|V(D)\setminus U_1|\leq(1/2+\varepsilon^\prime)n.$$
Hence, the claim holds.
\end{proof}
In the following, we will first prove that $|W_1|$ is $\varepsilon^\prime$-approximately equal to $|W_3|$, and similarly, $|W_2|$ is $\varepsilon^\prime$-approximately equal to $|W_4|$. This means that the sizes of $|W_1|$ and
$|W_3|$ differ by at most $O(\varepsilon^\prime n)$, and the sizes of $|W_2|$ and $|W_4|$ also differ by at most $O(\varepsilon^\prime n)$. These approximations are crucial for understanding the balance between the sizes of the vertex subsets $W_1$, $W_2$, $W_3$ and $W_4$ in the digraph $D$. We will now proceed to establish these relationships formally.
\begin{claim}\label{szw1}
$-3\varepsilon^\prime n/2+k\leq|W_1|-|W_3|\leq3\varepsilon^\prime n/2-k$ and $-\varepsilon^\prime n+2k\leq|W_2|-|W_4|\leq2\varepsilon^\prime n$.
\end{claim}
\begin{proof}
By Claim \ref{ccclaim}, we have for each $j\in\{1, 3\}$,
\begin{align*}
(1-\varepsilon^\prime)n/2+k-|W_2|\leq|W_j|\leq(1/2+\varepsilon^\prime)n-|W_2|.
\end{align*}
Hence, we deduce:
\begin{align}\label{yyyy1}
-3\varepsilon^\prime n/2+k\leq|W_1|-|W_3|\leq3\varepsilon^\prime n/2-k.
\end{align}
Also, by Claim \ref{ccclaim} again, we have that $(1/2-\varepsilon^\prime/2)n+k-|W_3|\leq|W_2|\leq(1/2+\varepsilon^\prime)n-|W_3|$, and
$(1/2-\varepsilon^\prime/2)n+k+|W_3|\leq|W_1|+|W_2|+|W_3|\leq (1/2+\varepsilon^\prime)n+|W_3|.$
Together with $n=|W_1|+|W_2|+|W_3|+|W_4|$, this suggests that
\begin{align*}
(1/2-\varepsilon^\prime)n-|W_3|\leq|W_4|\leq(1+\varepsilon^\prime)n/2-k-|W_3|.
\end{align*}
Hence, by claim 3.15, we obtain that
\begin{align}\label{yyyy2}
-\varepsilon^\prime n+2k\leq|W_2|-|W_4|\leq2\varepsilon^\prime n.
\end{align}
Inequalities (\ref{yyyy1}) and (\ref{yyyy2}) imply that $|W_1|$ $\varepsilon^\prime$-approximately equals $|W_3|$, and $|W_2|$ is $\varepsilon^\prime$-approximately equal to $|W_4|$; meaning their sizes differ by at most $O(\varepsilon^\prime n)$.
\end{proof}
We then estimate the cardinality of the vertex set $W_i$ for each $i\in[4]$. The following claim holds.
\begin{claim}\label{szw2} We declare that the statements hold as follows.\\
$(i)$ $\varepsilon_1n/2<|W_j|<(1/2-3\varepsilon_1/4)n\ \mbox{for}\ j\in\{1, 3\}.$\\
$(ii)$ $\varepsilon_1 n/2<|W_i|<(1/2-\varepsilon_1/4)n,\ \mbox{for each}\ i\in\{2, 4\}.$
\end{claim}
\begin{proof}
We first prove (i). Since $|W_4|=|U_0|$ and $\varepsilon_1n<|U_0|<(1/2-\varepsilon_1)n$, it is not hard to get that
\begin{align}\label{szw3}
\varepsilon_1n/2<|W_4|<(1/2-\varepsilon_1/4)n. \end{align}
In the following, we estimate the upper and lower bounds of $|W_j|$ for $j\in\{1, 3\}$. Since $|U_1|, |U_2| \geq(1/2-\varepsilon^\prime)n$ and $\varepsilon^\prime\ll\varepsilon_1$, we have: $$|W_j|\geq(1/2-\varepsilon^\prime)n-|W_4|>(1/2-\varepsilon^\prime)n-(1/2-\varepsilon_1)n\geq\varepsilon_1n/2.$$ Also by Claim \ref{ccclaim}, $V(D)=W_1\cup W_2\cup W_3\cup W_4$ and $|W_4|=|U_0|>\varepsilon_1n$, we can deduce:
\begin{align}\label{aaal}
|W_j|&=|V(D)|-(|W_{j+2}|+|W_2|)-|W_4|\nonumber\\
&< n-((1/2-\varepsilon^\prime/2)n+k)-\varepsilon_1 n\leq(1/2-3\varepsilon_1/4)n,
\end{align}
where the subscript of $W_{j+2}$ is taken modulo $4$. Combining these results, we conclude:
\begin{align}\label{W1}
\varepsilon_1n/2<|W_j|<(1/2-3\varepsilon_1/4)n\ \mbox{for}\ j\in\{1, 3\}.
\end{align}

\smallskip

Next, we prove (ii). We now estimate the upper and lower bounds of $|W_2|$. On the one hand, since $|U_1|\geq(1/2-\varepsilon^\prime)n$, $|W_2|=|V(D)|-(|W_1|+|W_3|+|W_4|)=|V(D)|-|U_1\cup U_2|$, we have that
\begin{align*}\label{eqqq1}
|W_2|=n-(|U_1|+|U_{2}\setminus U_0|)
<n-((1/2-\varepsilon^\prime)n+\varepsilon_1n/2)<(1/2-\varepsilon_1/4)n.
\end{align*}
On the other hand, by Claim \ref{ccclaim} again and (\ref{aaal}), we obtain that
\begin{align*}
|W_2|\geq(1/2-\varepsilon^\prime/2)n+k-|W_1|>(1/2-\varepsilon^\prime/2)n+k-(1/2-3\varepsilon_1/4)n\geq\varepsilon_1 n/2.
\end{align*}
Together with (\ref{szw3}), we conclude that
\begin{equation}\label{W2}
\begin{split}
\varepsilon_1 n/2<|W_i|<(1/2-\varepsilon_1/4)n,\ \mbox{for each}\ i\in\{2, 4\}.
\end{split}
\end{equation}
Therefore, the claim is proven.
\end{proof}
In what follows, we will prove that vertex sets $W_1$, $W_2$, $W_3$ and $W_4$ satisfy properties $(A)$, $(B)$ and $(C)$ of \textbf{EC1}. Firstly, it follows from $e^+(U_1, U_2)\leq(\varepsilon^\prime n)^2$ that, for each $j\in\{1, 4\}$, $e^+(W_j, W_3\cup W_4)\leq e^+(U_1, U_2)\leq(\varepsilon^\prime n)^2$. By Claim \ref{ccclaim}, $\delta^0(D)\geq n/2+k$ and $\varepsilon^\prime\ll \varepsilon$, this implies that for $j\in \{1,4\}$
\begin{align}\label{W3}
e^+(W_j, W_1\cup W_2)&\geq|W_j|\cdot(n/2+k)-(\varepsilon^\prime n)^2\nonumber\\
&=|W_j|\cdot(1/2+\varepsilon^\prime)n+|W_1|\cdot(k-\varepsilon^\prime n)-(\varepsilon^\prime n)^2\nonumber\\
&\geq|W_j|\cdot(|W_1|+|W_2|)-2\varepsilon^\prime n^2.
\end{align}

\smallskip

\noindent Secondly, by Claim \ref{ccclaim}, and since for each $j\in\{3, 4\}$, $e^+(W_1\cup W_4, W_j)\leq e^+(U_1, U_2)\leq(\varepsilon^\prime n)^2$, we get that
\begin{align}\label{W5}
e^+(W_2\cup W_3, W_j)&\geq(n/2+k)\cdot|W_j|-(\varepsilon^\prime n)^2\nonumber\\
&=(1/2+\varepsilon^\prime)n\cdot|W_j|+(k-\varepsilon^\prime n)\cdot|W_3|-(\varepsilon^\prime n)^2\nonumber\\
&\geq(|W_2|+|W_3|)\cdot|W_j|-2\varepsilon^\prime n^2.
\end{align}
Inequality equations (\ref{W3}) and (\ref{W5}) suggest that the vertex sets $W_1$, $W_2$, $W_3$ and $W_4$ of $D$ satisfy properties $(A)$, $(B)$ and $(C)$ of \textbf{EC1}. Together with Claims \ref{szw1} and \ref{szw2}, and by $\varepsilon^\prime\ll\varepsilon\ll1$, we have shown that $D$ belongs to \textbf{EC1} (see Figure \ref{fig1} $(c)$).

Hence the proof of the lemma is completed.
\end{proof}
\emph{Remark.} We can get from the proof of Lemma \ref{claim1} that if $|U_1\cap U_2|\leq\varepsilon_1n$, where vertex sets $U_1$ and $U_2$ satisfy the extrenal condition \textbf{EC} with parameter $\varepsilon^\prime$, then $|W_1|, |W_3|\geq(1/2-\varepsilon/2)n$ and $e(W_i)\geq|W_i|^2-\varepsilon n^2$ for each $i\in\{1, 3\}$; if $|U_1\cap U_2|\geq(1/2-\varepsilon_1)n$, then $|W_2|=(1/2-\varepsilon_1)n$ and $|W_4|=(1/2+\varepsilon_1)n$ and $e^+(W_2, W_{4}), e^+(W_4, W_{2}) \geq|W_2|\cdot|W_{4}|-\varepsilon n^2$; otherwise, that is, $\varepsilon_1n<|U_1\cap U_2|<(1/2-\varepsilon_1)n$, then $\frac{\varepsilon_1n}{2}\leq|W_i|\leq(\frac{1}{2}-\frac{\varepsilon_1}{4})n$ for each $i\in[4]$ and $e^+(W_i, W_{i+1})\geq|W_i|\cdot|W_{i+1}|-\frac{\varepsilon^\prime n}{2}$, where $W_5=W_1$.

\medskip

By Lemma \ref{claim1}, we know that $D$ belongs to \textbf{EC1} if it is not stable. In the following, we provide the proof of Theorem \ref{song2} based on the case when $D$ belongs to \textbf{EC1}. 
Recall that we have $\delta^0(D)\geq n/2+k$, and our goal is to show that $D$ is $(\alpha, \beta)$-arbitrary
Hamiltonian $H$-linked. In the rest of this paper, we also use the vertex set to represent the subgraph induced by it for simplicity.
\begin{lemma}\label{case}
Let $H$ be a digraph with $k$ arcs and $\delta(H)\geq1$. Let $C_0$ be a constant, and parameters $\alpha, \beta$ satisfy $1/C_0\ll\alpha, \beta\ll1$. If $D$ is a digraph of order $n\geq C_0k$ with $\delta^0(D)\geq n/2+k$ and belongs to \textbf{EC1}, then $D$ is $(\alpha, \beta)$-arbitrary Hamiltonian $H$-linked.
\end{lemma}
\begin{proof}
%
%
Let $F=f(V(H))$ for brevity, and define $W_i^\prime=W_i\setminus F$ for each $i\in[4]$.
From Definition \ref{ec} and Lemma \ref{claim1}, we have that $|W_1^\prime|$ is $\varepsilon^\prime$-approximately equal to $|W_3^\prime|$, and similarly, $|W_2^\prime|$ is $\varepsilon^\prime$-approximately equal to $|W_4^\prime|$. Using properties $(A)$-$(C)$ of \textbf{EC1}, we observe that $(W_2^\prime, W_4^\prime)$ forms an $\varepsilon$-almost complete bipartite, and $D[W_i^\prime]$ (for $i\in\{1, 3\}$) is $\varepsilon$-almost complete. In what follows, we proceed by case analysis on the cardinality of $U_1\cap U_2$, where $U_1$ and $U_2$ satisfy the extremal condition \textbf{EC} with parameter $\varepsilon^\prime$.
\begin{Case}\label{case3.1}
$|U_1\cap U_2|\leq\varepsilon_1n$.
\end{Case}
In this case, we complete the proof of Lemma \ref{case} in three steps:

\smallskip

\noindent\textbf{Step 1.} For every vertex pair $(v_i, v_i^\prime)$ ($i\in[k]$), show there exists $j\in\{1, 3\}$ such that

$\bullet$ a $(v_i, W_j^\prime)$-path and a $(W_j^\prime, v_i^\prime)$-path exist (Claim \ref{ccccl});

\smallskip

\noindent\textbf{Step 2.} For $j\in\{1, 3\}$, use disjoint $W_1$-paths and $W_3$-paths to cover low-semi-degree vertices in $W_j^\prime$ (Claim \ref{S}), which serve as subpaths for $(v_i, v_i^\prime)$-paths ($i\in[k]$);

\smallskip

\noindent\textbf{Step 3.} In the subdigraph induced by tne remaining of $W_j^\prime$ where $j\in\{1, 3\}$, apply Proposition \ref{prop} to construct disjoint paths of desired lengths. Combine these with paths from Steps 1-2 to form the final $(v_i, v_i^\prime)$-paths for all $i\in[k]$.
\begin{proof}
Before proceeding to \textbf{Step 1}, we do some preparatory work. Let $R=V(D)\setminus(F\cup W_1^\prime\cup W_3^\prime)$. By Definition \ref{definition3.2} with $(U_1, U_2)_{\ref{definition3.2}}=(W_1^\prime, W_3^\prime)$, and \textbf{EC1}, we get that for each $i\in\{1, 3\}$,
$$e(W_i^\prime)\geq|W_i^\prime|^2-3\varepsilon n^2\Rightarrow |E_2\cap W_i^\prime|\leq|E_1\cap W_i^\prime|\leq\sqrt{10\varepsilon}|W_i^\prime|.$$ Further, if there exists a vertex $x\in E_2\cap W_1^\prime\cup R$ (resp., a vertex $y\in E_2\cap W_3^\prime \cup R$) such that for each $\sigma\in\{+, -\}$, $d_{W_3^\prime}^\sigma(x)>\varepsilon^{1/3}|W_3^\prime|)$ (resp., $d_{W_1^\prime}^\sigma(y)>\varepsilon^{1/3}|W_1^\prime|)$, then we move $x$ (resp., $y$) into the vertex set $W_3^\prime$ (resp., $W_1^\prime$) and update the vertex sets $W_3^\prime$ and $W_1^\prime$. We repeat the above operation until there are no such vertices $x$ and $y$. Note that this process is guaranteed to terminate after a finite number of steps, because both $|E_2\cap W_1^\prime\cup R|$ and $|E_2\cap W_3^\prime \cup R|$ are very small.

\smallskip

Next, define $S_1=W_1^\prime\setminus E_2$, $S_2=W_3^\prime\setminus E_2$, and $S_3=V(D)\setminus(F\cup S_1\cup S_2)$. Clearly, $|S_1|, |S_2|\geq(1/2-\varepsilon/2-\sqrt{10\varepsilon})n$. Using the lower bound of $\delta^0(D)$ and the definitions of $S_1$, $S_2$ and $S_3$, it is straightforward to verify the following properties.
\begin{itemize}\item[$(\mathcal{A}1)$] \emph{For each $i\in[2]$, there exists a subset $S_i^\prime\subseteq S_i$ with $|S_i^\prime|\leq 10\sqrt{\varepsilon}|S_i|$ such that \\
$\bullet$ for every vertex $v\in S_i\setminus S_i^\prime$, $\delta^0_{S_i}(v)\geq (1-10\sqrt{\varepsilon})|S_i|$, and\\
$\bullet$ for every vertex $v\in S_i^\prime$, $\delta^0_{S_i}(v)\geq\frac{\varepsilon^{1/3}|S_i|}{2}$.}
\end{itemize}
\begin{itemize}\item[$(\mathcal{A}2)$] \emph{Furthermore,  for every vertex $v\in S_3$, one of the following holds: either $d^-_{S_1}(v), d^+_{S_2}(v)\\ >\frac{(1-2\varepsilon^{1/3})n}{2}$, or $d^+_{S_1}(v), d^-_{S_2}(v)>\frac{(1-2\varepsilon^{1/3})n}{2}$. Also, $\delta^0_{S_i}(v)\leq\frac{\varepsilon^{1/3}n}{2}$ for each $i\in[2]$.}
\end{itemize}

Sequentially, by the semi-degree condition of $D$ and the cardinalities of $W_1^\prime$ and $W_3^\prime$, we deduce that for any vertex $v_i\in F$, it is connected to and from many vertices in $S_1\cup S_2$. Furthermore, by Definition \ref{definition3.3} with $(U_1, U_2)_{\ref{definition3.3}}=(S_1, S_2)$, it is clear that each pair $(v_i, v_{i}^\prime)$ ($i\in[k]$) belongs to $V_j$ for some $j\in[4]$. Therefore, we now construct a partition $(V_1^\prime, V_2^\prime, V_3^\prime, V_4^\prime)$ of pairs $(v_i, v_{i}^\prime)$ $(i\in[k])$ with $V_i^\prime\subseteq V_i$ for $i\in[4]$, such that

\emph{(\uppercase\expandafter{\romannumeral1}) $|V_1^\prime\cup V_2^\prime|$ is as large as possible, and }

\emph{(\uppercase\expandafter{\romannumeral2}) subject to (\uppercase\expandafter{\romannumeral1}), the quantity $|\sum_{(v_i, v_{i}^\prime)\in V_1^\prime}n_i-\sum_{(v_j, v_{j}^\prime)\in V_2^\prime}n_j|$ is as small as possible.}

\smallskip

\noindent For convenience, we label the partition as follows:
\begin{equation*}
\begin{split}
&V_1^\prime=\{(v_1, v_{1}^\prime), \ldots, (v_{l_1}, v_{l_1}^\prime)\},\\
&V_2^\prime=\{(v_{l_1+1}, v_{l_1+1}^\prime), \ldots, (v_{l_1+l_2}, v_{l_1+l_2}^\prime)\},\\
&V_3^\prime=\{(v_{l_1+1_2+1}, v_{l_1+l_2+1}^\prime), \ldots, (v_{l_1+l_2+l_3}, v_{l_1+l_2+l_3}^\prime)\},\\
&V_4^\prime=\{(v_{l_1+l_2+l_3+1}, v_{l_1+l_2+l_3+1}^\prime), \ldots, (v_{l_1+l_2+l_3+l_4}, v_{l_1+l_2+l_3+l_4}^\prime)\}.\\
\end{split}
\end{equation*}
Here, $|V_i^\prime|=l_i$ for each $i\in[4]$, and $l_1+l_2+l_3+l_4=k$. Further, we define the following threshold conditions $(a)$ and $(b)$:
\begin{equation}\label{ab}
\begin{split}
(a)\ |S_1|<n_1+\cdots+n_{l_1};\ (b)\ |S_2|<n_{l_1+1}+\cdots +n_{l_1+l_2}.
\end{split}
\end{equation}
These conditions will be used in the subsequent steps of the proof. We now proceed to \textbf{Step 1} with the following key assertion.
\begin{claim}\label{ccccl}
For every vertex pair $(v_i, v_{i}^\prime)$ \emph{(}$i\in[k]$\emph{)}, there exists a set $S_j$ with $j\in[2]$ and two disjoint paths of length at most $4$, such that one path is from $v_i$ to $S_j$ and the other is from $S_j$ to $v_{i}^\prime$.
\end{claim}
\begin{proof}
We prove this claim by considering two subcases as follows:
\begin{Subcase}\label{subcase1}
At least one of conditions $(a)$ and $(b)$ in $(\ref{ab}) $ is true.
\end{Subcase}
In this subcase, we only give the proof of the case when $(a)$ holds but $(b)$ does not hold, since we can similarly solve the case when $(b)$ is true but $(a)$ is not true (the symmetric case), and the case that both $(a)$ and $(b)$ hold. So, we omit their proofs.

\smallskip

Without loss of generality, we can assume that $n_1\geq n_2\geq\cdots\geq n_{l_1}$. Let $i_0\in[l_1]$ be the minimal index satisfying:
$$n_{i_0+1}+\cdots+n_{l_1}\leq|S_1|< n_{i_0}+\cdots+ n_{l_1}.$$ Define $V_1^{\prime\prime}=V_1^\prime\setminus\{(v_1, v_{1}^\prime), \ldots, (v_{i_0}, v_{i_0}^\prime)\}$. Then, we have the key observations:

(i) For each vertex pair $(v_i, v_i^\prime)$ in $V_1^{\prime\prime}$, there are two disjoint arcs: one is from $v_i$ to $S_1$ and the other is from $S_1$ to $v_i^\prime$;

(ii) For every vertex pair $(v_i, v_i^\prime)$ in $V_3^\prime$, there is an arc from $S_2$ to $v_i^\prime$;

(iii) For each $(v_i, v_i^\prime)$ in  $V_4^\prime$, there exists an arc from $v_i$ to $S_2$.

\noindent Further, we affirm the following conclusions.

\smallskip

\emph{\textbf{\emph{(}3.1.1\emph{)}} We can construct a set $\mathcal{P}_1$ of disjoint minimal paths of length at most $4$ in $D$ with $|\mathcal{P}_1|=2i_0+l_3+l_4$ and $|V(P)\cap S_3|\leq1$ for each $P\in\mathcal{P}_1$, such that}

\emph{\textbf{\emph{(}D1\emph{)}} for every vertex pair $(v_i, v_i^\prime)$ $(i\in[i_0])$, there are two disjoint paths in $\mathcal{P}_1$: one is from $v_{i}$ to $S_2$, and the other is from $S_2$ to $v_{i}^\prime$;}

\emph{\textbf{\emph{(}D2\emph{)}} for every vertex pair $(v_i, v_i^\prime)$ in $V_3^\prime$, there is a path of from $v_{i}$ to $S_2$;}

\emph{\textbf{\emph{(}D3\emph{)}} for each vertex pair $(v_i, v_i^\prime)$ in $V_4^\prime$, there exists a path from $S_2$ to $v_{i}^\prime$.}
\begin{proof}
Assume $|\mathcal{P}_1|<2i_0+l_3+l_4$, with $|\mathcal{P}_1|$ maximized and $|V(\mathcal{P}_1)|$ minimized. First we assume that $(D1)$ is incorrect. Consider two cases:

\smallskip

\textbf{Case A (Insufficient $(v_i, S_2)$-paths for $i\in[i_0]$).} If the number of disjoint $(v_i, S_2)$-paths is less than $i_0$, then consider another vertex pair $(v_j, v_{j}^\prime)$ with $j\in[i_0]$. If an arc exists from $N_{S_1\setminus V(\mathcal{P}_1)}^+(v_j)$ to $S_2$, then this contradicts the maximality of $|\mathcal{P}_1|$. Moreover, if there exists a vertex $b_{j}$ in $S_2\setminus V(\mathcal{P}_1)$ with $d^-_{S_1\setminus V(\mathcal{P}_1)}(b_{j})\neq 0$, then by $(\mathcal{A}1)$ a path of length at most $2$ exists from $N_{S_1\setminus V(\mathcal{P}_1)}^+(v_j)$ to $N_{S_1\setminus V(\mathcal{P}_1)}^-(b_{j})$. Thus, in $D-\mathcal{P}_1$, a path of length at most $4$ exists from $v_j$ to $S_2\setminus V(\mathcal{P}_1)$, contradicting the maximality of $|\mathcal{P}_1|$. Symmetrically, if there exists a vertex $a_j\in S_1\setminus V(\mathcal{P}_1)$ with $d^+_{S_2\setminus V(\mathcal{P}_1)}(a_{j})\neq 0$, then by $(\mathcal{A}1)$ again, there is a path of length at most $2$ from $N_{S_1\setminus V(\mathcal{P}_1)}^+(v_j)$ to $N^-_{S_1\setminus V(\mathcal{P}_1)}(a_{j})$. This implies a path of length at most $4$ from $v_j$ to $S_2\setminus V(\mathcal{P}_1)$, again contradicting the maximality of $|\mathcal{P}_1|$.

\smallskip

Hence, we reduce the case where
$$ \mbox{for any}\ a_j\in S_1\setminus V(\mathcal{P}_1)\ \mbox{and any}\ b_{j}\in S_2\setminus V(\mathcal{P}_1),\  d^+_{S_2\setminus V(\mathcal{P}_1)}(a_j)=0=d^-_{S_1\setminus V(\mathcal{P}_1)}(b_{j}).$$ Obviously, $d^-_{V_1^{\prime\prime}}(b_{j})\leq i_0$, because otherwise it would contradict the maximality of $|\mathcal{P}_1|$. Also, by (\uppercase\expandafter{\romannumeral1}) and (\uppercase\expandafter{\romannumeral2}), for at least $n/3$ such vertices $a_j$ and $b_{j}$, respectively, we have that $$d^+_{V_3^\prime}(a_j)\leq l_3,\ \mbox{and}\ d^-_{V_3^\prime}(b_{j})\leq l_3.$$ Since, otherwise, we can get that there are least $n/6$ vertices $a_i\in S_1\setminus V(\mathcal{P}_1)$ with $d^+_{V_3^\prime}(a_j)>l_3$, or at least $n/6$ vertices $b_{j}\in S_2\setminus V(\mathcal{P}_1)$ with $d^-_{V_3^\prime}(b_{j})>l_3$, respectively. This further suggests that there exists a vertex pair $(v_{j}, v_{j}^\prime)$ in $V_3^\prime$ such that $$d^+_{S_1}(v_{j}), d^-_{S_1}(v_{j})\geq 4k,\ \mbox{or}\ d^+_{S_2}(v_{j}), d^-_{S_2}(v_{j})\geq 4k,\ \mbox{respectively}.$$ Then we have $V_3^\prime\cap(V_1\cup V_2)\neq\emptyset$, a contradiction with (\uppercase\expandafter{\romannumeral1}) above. So $$d^+_F(a_j)\leq2l_1+2l_2+l_3+2l_4\ \mbox{and}\ d_F^-(b_{j})\leq2(l_1-i_0)+i_0+2l_2+l_3+2l_4.$$ Then using  $\delta^0(D)\geq n/2+k$, the minimality of $|V(\mathcal{P}_1)|$ and $(\mathcal{A}2)$, we get that for at least $n/3$ such vertices $a_j$ and $b_{j}$, respectively,
\begin{align}\label{erer}
&|N^+_{S_3\setminus V(\mathcal{P}_1)}(a_j)\cap N^-_{S_3\setminus V(\mathcal{P}_1)}(b_{j})|\nonumber\\
&\geq2(n/2+k)-(|S_1|+|S_2|+|S_3\setminus V(\mathcal{P}_1)|-(4l_1-i_0+4l_2+2l_3+4l_4)-|\mathcal{P}_1|\nonumber\\
&\geq i_0+2l_3.
\end{align}
This implies the existence of another $(v_i, S_2)$-path, contradicting the maximality of $|\mathcal{P}_1|$.

\medskip

\textbf{Case B (Insufficient $(S_2, v_i^\prime)$-paths for $i\in[i_0]$).} If the number of disjoint $(S_2, v_i^\prime)$-paths is less than $i_0$, then consider another vertex pair $(v_j, v_{j}^\prime)$ with $j\in[i_0]$. If an arc exists from $S_2$ to $N_{S_1\setminus V(\mathcal{P}_1)}^-(v_j^\prime)$, then this contradicts the maximality of $|\mathcal{P}_1|$. More generally, if there exists a vertex $b_{j}\in S_2\setminus V(\mathcal{P}_1)$ with $d^+_{S_1\setminus V(\mathcal{P}_1)}(b_{j})\neq 0$, then by $(\mathcal{A}1)$ a path of length at most $2$ exists from $N_{S_1\setminus V(\mathcal{P}_1)}^+(b_j)$ to $N_{S_1\setminus V(\mathcal{P}_1)}^-(v_{j}^\prime)$. Thus, in $D-\mathcal{P}_1$ there is a path of length at most $4$ from $S_2\setminus V(\mathcal{P}_1)$ to $v_j^\prime$, contradicting the maximality of $|\mathcal{P}_1|$, again. Symmetrically, if there exists a vertex $a_j\in S_1\setminus V(\mathcal{P}_1)$ with $d^-_{S_2\setminus V(\mathcal{P}_1)}(a_{j})\neq 0$, then by $(\mathcal{A}1)$ again, a path of length $\leq2$ exists from $N_{S_1\setminus V(\mathcal{P}_1)}^+(a_j)$ to $N^-_{S_1\setminus V(\mathcal{P}_1)}(v_j^\prime)$. This implies a path of length at most $4$ exists from $S_2\setminus V(\mathcal{P}_1)$ to $v_j^\prime$, again contradicting the maximality of $|\mathcal{P}_1|$.

\smallskip

Hence, we reduce the case where
$$\mbox{for any}\ a_j\in S_1\setminus V(\mathcal{P}_1)\ \mbox{and any}\ b_{j}\in S_2\setminus V(\mathcal{P}_1),\ d^-_{S_2\setminus V(\mathcal{P}_1)}(a_j)=0=d^+_{S_1\setminus V(\mathcal{P}_1)}(b_{j}).$$ Then symmetric to $(\ref{erer})$, for at least $n/3$ such vertices $a_j$ and $b_j$, respectively,
\begin{equation*}
\begin{split}
&|N^-_{S_3\setminus V(\mathcal{P}_1)}(a_j)\cap N^+_{S_3\setminus V(\mathcal{P}_1)}(b_{j})|\\
&\geq2(n/2+k)-(|S_1|+|S_2|+|S_3\setminus V(\mathcal{P}_1)|)-(4l_1-i_0+4l_2+4l_3+2l_4)-|\mathcal{P}_1|\\
&\geq i_0+2l_4.
\end{split}
\end{equation*}
This implies the existence of another $(S_2, v_j^\prime)$-path, contradicting the maximality of $|\mathcal{P}_1|$.

\smallskip

Finally, the case where (D2) fails is analogous to Case A. Similarly, the case where (D3) fails is analogous to Case B. Since the proofs for these cases follow the same reasoning as previously established, we omit them for brevity. Thus Subcase \ref{subcase1} holds.
\end{proof}
\begin{Subcase}\label{subcase2}
Neither condition $(a)$ nor $(b)$ of $(\ref{ab})$ holds.
\end{Subcase}
In this subcase, for every vertex pair $(v_i, v_i^\prime)$ in $V_j^\prime$ ($j\in[2]$), there exist two disjoint arcs: one from $v_i$ to $S_j$ and another from $S_j$ to $v_i^\prime$. Without loss of generality, assume there exists a subscript $i_1$ ($l_1+l_2+1\leq i_1\leq l_1+l_2+l_3$) such that  $$n_{l_1+l_2+1}+\cdots+n_{i_1}\leq|S_1|-\sum\nolimits_{i=1}^{l_1}n_i< n_{l_1+l_2+1}+\cdots+n_{i_1+1}.$$ Then we now assert the following statement.

\smallskip

\emph{\textbf{\emph{(}3.1.2\emph{)}} There is a set $\mathcal{P}_2$ of $l_3+l_4$ disjoint minimal paths $($each of length $\leq4$$)$ with $|V(P)\cap S_3|\leq1$ for all $P\in \mathcal{P}_2$, such that}

\emph{\textbf{\emph{(}F1\emph{)}} for every vertex pair $(v_i, v_{i}^\prime)$ with $l_1+l_2+1\leq i\leq i_1$, $\mathcal{P}_2$ contains a $(S_1, v_i^\prime)$-path;}

\emph{\textbf{\emph{(}F2\emph{)}} for each $(v_i, v_{i}^\prime)$ with $i_1< i\leq l_1+l_2+l_3$, $\mathcal{P}_2$ contains a $(v_i, S_2)$-path;}

\emph{\textbf{\emph{(}F3\emph{)}} for each $(v_i, v_{i}^\prime)$ with $l_1+l_2+l_3+1\leq i\leq k$, $\mathcal{P}_2$ contains a $(S_2, v_i^\prime)$-path.}
\begin{proof}
Assume $|\mathcal{P}_2|<l_3+l_4$, with $|\mathcal{P}_2|$ maximized and $|V(\mathcal{P}_2)|$ minimized. We first assume that (F1) fails. For another vertex pair $(v_j, v_{j}^\prime)$ with $l_1+l_2+1\leq j\leq i_1$, if there exists an arc from $S_1\setminus V(\mathcal{P}_2)$ to $N^-_{S_2\setminus V(\mathcal{P}_2)}(v_{j}^\prime)$, or there exists a vertex $a_{j}\in S_1\setminus V(\mathcal{P}_2)$ with $d^+_{S_2\setminus V(\mathcal{P}_2)}(a_j)\neq0$, and there exists a vertex $b_{j}\in S_2\setminus V(\mathcal{P}_2)$ with $d^-_{S_1\setminus V(\mathcal{P}_2)}(b_j)\neq0$, respectively, then by $(\mathcal{A}1)$, there exists a path of length at most $2$ from $N^+_{S_2\setminus V(\mathcal{P}_2)}(a_j)$ to $N^-_{S_2\setminus V(\mathcal{P}_2)}(v_j^\prime)$, and from $b_j$ to $N^-_{S_2\setminus V(\mathcal{P}_2)}(v_j^\prime)$, respectively. This implies a $(S_1, v_j^\prime)$-path in $D-\mathcal{P}_2$, contradicting the maximality of $|\mathcal{P}_2|$.

 \smallskip


\smallskip

Thus, we reduce to the case that for all $a_{j}\in S_1\setminus V(\mathcal{P}_2)$ and $b_{j}\in S_2\setminus V(\mathcal{P}_2)$:
$$d^+_{S_2\setminus V(\mathcal{P}_2)}(a_{j})=0\ \mbox{and}\ d^-_{S_1\setminus V(\mathcal{P}_2)}(b_{j})=0.$$
Obviously, we have that $d^+_F(a_{j})\leq2l_1+2l_2+l_3+2l_4$ and $d^-_{F}(b_{j})\leq2l_1+2l_2+l_3+2l_4$. Since, otherwise, it will contradict with the choosing condition (\uppercase\expandafter{\romannumeral1}) above. Further, by the lower bound of $\delta^0(D)$, the minimality of $|V(\mathcal{P}_2)|$, and $(\mathcal{A}2)$, for at least $n/3$ such vertices $a_{j}$ and $b_{j}$, respectively, we obtain that
\begin{align}\label{ererer}
&|N^+_{S_3\setminus V(\mathcal{P}_2)}(a_j)\cap N^-_{S_3\setminus V(\mathcal{P}_2)}(b_{j})|\nonumber\\
&\geq2(n/2+k)-(|S_1|+|S_2|+|S_3\setminus V(\mathcal{P}_2)|)-(4l_1+4l_2+2l_3+4l_4)-|\mathcal{P}_2|\nonumber\\
&\geq2l_3.
\end{align}
This implies another $(S_1, v_j^\prime)$-path exists, contradicting the maximality of $|\mathcal{P}_2|$ again.

\smallskip

The proof for the case when (F2) fails is analogous, and so is omitted. For (F3), if it fails, then in the same way, we consider another vertex pair $(v_j, v_{j}^\prime)$ with $l_1+l_2+l_3+1\leq j\leq k$. If an arc exists from $S_2\setminus V(\mathcal{P}_2)$ to $N^-_{S_1\setminus V(\mathcal{P}_2)}(v_{j}^\prime)$, or there exists a vertex $b_{j}$ in $S_2\setminus V(\mathcal{P}_2)$ with $d^+_{S_1\setminus V(\mathcal{P}_2)}(b_j)\neq0$, and there is a vertex $a_{j}\in S_1\setminus V(\mathcal{P}_2)$ with $d^-_{S_2\setminus V(\mathcal{P}_2)}(a_{j})\neq0$, respectively, then another $(S_2, v_j^\prime)$-path exists: by $(\mathcal{A}1)$, a path of length at most $2$ exists from $N^+_{S_1\setminus V(\mathcal{P}_2)}(b_j)$ to $N^-_{S_1\setminus V(\mathcal{P}_2)}(v_j^\prime)$, and from $a_j$ to $N^-_{S_1\setminus V(\mathcal{P}_2)}(v_j^\prime)$, respectively. This contradicts the maximality of $|\mathcal{P}_2|$.

\smallskip

Hence we come down to the case that
$$\mbox{for all}\ a_j\in S_1\setminus V(\mathcal{P}_2)\ \mbox{and}\ b_{j}\in S_2\setminus V(\mathcal{P}_2),\ d^-_{S_2\setminus V(\mathcal{P}_2)}(a_j)=0=d^+_{S_1\setminus V(\mathcal{P}_2)}(b_j).$$ Symmetrically to the first case, we get that $d^-_F(a_j)\leq2l_1+2l_2+2l_3+l_4$ and $d^+_F(b_j)\leq2l_1+2l_2+2l_3+l_4$. So similar to $(\ref{ererer})$, for at least $n/3$ such vertices $a_j$ and $b_j$, we derive:
\begin{equation*}
\begin{split}
&|N^-_{S_3-2|\mathcal{P}_2|}(a_j)\cap N^+_{S_3-2|\mathcal{P}_2|}(b_j)|\\
&\geq2(n/2+k)-(|S_1|+|S_2|+|S_3\setminus V(\mathcal{P}_2)|)-(4l_1+4l_2+4l_3+2l_4)-|\mathcal{P}_2|\\
&\geq2l_4.
\end{split}
\end{equation*}
This suggests that we can obtain another $(S_2, v_j^\prime)$-path, again contradicting the maximality of $|\mathcal{P}_2|$. Thus, the statement (3.1.2) holds, and then Subcase \ref{subcase2} is proved.
\end{proof}

\smallskip

Combining Subcases \ref{subcase1} and \ref{subcase2}, we conclude that Claim \ref{ccccl} holds.
\end{proof}
To complete \textbf{Step 2}, we need to establish Claim \ref{S}. Let $\mathcal{P}$ be the set of  disjoint paths obtained in Claim \ref{ccccl}. For each $i\in[3]$, define $S_i^\prime=S_i\setminus V(\mathcal{P})$. We show the following conclusion.
\begin{claim}\label{S}
For any vertex $u$ in $S_3^\prime$, there exists some subscript $j\in[2]$ such that there is a $S_{j}^\prime$-path of length at most $4$ containing the vertex $u$.
\end{claim}
\begin{proof}
Let $S_{3, 1}^\prime$ (resp., $S_{3, 2}^\prime$) be the set of vertices $u$ in $S_3^\prime$ that satisfy $d^-_{S_1}(u), d^+_{S_2}(u)>\frac{1}{2}(1-2\varepsilon^{1/3})n$ (resp., $d^+_{S_1}(u), d^-_{S_2}(u)>\frac{1}{2}(1-2\varepsilon^{1/3})n$). 
We first consider the case when $|S_{3, 1}^\prime|=|S_{3, 2}^\prime|$. For any $u\in S_{3, 1}^\prime$ and any vertex $v\in S_{3, 2}^\prime$, the intersection properties yield: 
\begin{equation*}
\begin{split}
&|N^-_{S_1^\prime}(u)\cap N^+_{S_1^\prime}(v)|\geq(1-2\varepsilon^{1/3})n-(|S_1^\prime|+2|V(\mathcal{P})|)>n/3,\ and\\
 &|N^+_{S_2^\prime}(u)\cap N^-_{S_2^\prime}(v)|\geq(1-2\varepsilon^{1/3})n-(|S_2^\prime|+2|V(\mathcal{P})|)>n/3.
\end{split}
\end{equation*}
This implies the existence of $|S_{3, 1}^\prime|$ disjoint $S_2^\prime$-paths (or $S_1^\prime$-paths) of length $4$ and the form $S_2^\prime\rightarrow v\rightarrow S_1^\prime\rightarrow u\rightarrow S_2^\prime$ (or $S_1^\prime\rightarrow u\rightarrow S_2^\prime\rightarrow v\rightarrow S_1^\prime$), covering $S_{3, 1}^\prime\cup S_{3, 2}^\prime$. The claim holds in this case.

\smallskip

Now suppose $|S_{3, 1}^\prime|\neq|S_{3, 2}^\prime|$. Without loss of generality, suppose that $|S_{3, 1}^\prime|>|S_{3, 2}^\prime|$ and let $r=|S_{3, 1}^\prime|-|S_{3, 2}^\prime|$. Next, we define $M_1$ (resp., $M_2$) as the set of matching edges from $S_{3, 1}^\prime$ to $S_1^\prime$ (resp., from $S_{2}^\prime$ to $S_{3, 1}^\prime$), and $M_1$ and $M_2$ are disjoint. We take $M_i$ for any $i\in[2]$ to be as large as possible such that $|S_{3, 1}^\prime\setminus V(M)|$ is minimum. Set $M=M_1\cup M_2$.

\smallskip

If $|M|\geq r$, then for any vertex $u\in S_{3, 1}^\prime$, the degree conditions $d^-_{S_1}(u), d^+_{S_2}(u)>\frac{1}{2}(1-2\varepsilon^{1/3})n$ ensure the existence of $r$ disjoint $S_1^\prime$-paths and $S_2^\prime$-paths of length $2$, of the form $S_1\rightarrow M_1$ and $M_2\rightarrow S_2$, respectively. Clearly, $|S_{3, 1}^\prime|-r=|S_{3, 2}^\prime|$. For the remaining vertices $u\in S_{3, 1}^\prime$ and all $v\in S_{3, 2}^\prime$, the intersection argument (analogous to the previous case) provides $|S_{3, 2}^\prime|$ disjoint $S_2^\prime$-paths (or $S_1^\prime$-paths ) of length $4$, of the form $$S_2^\prime\rightarrow v\rightarrow S_1^\prime\rightarrow u\rightarrow S_2^\prime\ (\mbox{or}\ S_1^\prime\rightarrow u\rightarrow S_2^\prime\rightarrow v\rightarrow S_1^\prime),$$ covering all remaining vertices of $S_{3, 1}^\prime$ and all vertices in $S_{3, 2}^\prime$. This completes the proof of Claim \ref{S}.

\smallskip

In the following, we assume $|M|<r$, and define $S_{i}^{\prime\prime}=S_{i}^\prime\setminus V(M)$ and $S_{3, 1}^{\prime\prime}=S_{3, 1}^\prime\setminus V(M)$. Since $M_1$ and $M_2$ are maximum matchings and $n=|S_1^\prime|+|S_2^\prime|+|S_{3, 1}^\prime|+|S_{3, 2}^\prime|+|V(\mathcal{P})|$, for any vertex $u\in S_1^{\prime\prime}$, we have that
\begin{align}\label{313}
d^-_{S_2^{\prime\prime}}(u)&\geq\delta^0(D)-(|S_1^{\prime}\setminus V(M)|+|S_{3, 2}^{\prime}|+d^-_{M}(u)+d^-_{\mathcal{P}}(u))\nonumber\\
&\geq\frac{|S_2^\prime|-|S_1^\prime|+|S_{3, 1}^\prime|-|S_{3, 2}^\prime|+|V(\mathcal{P})|}{2}+k+|M_1|-d^-_{M}(u)-d^-_{\mathcal{P}}(u).
\end{align}
Symmetrically, for any vertex $v$ in $S_{2}^{\prime\prime}$, we obtain that
 \begin{align}\label{314}
d^+_{S_1^{\prime\prime}}(v)&\geq\delta^0(D)-(|S_2^{\prime}\setminus V(M)|+|S_{3, 2}^{\prime}|+d^+_{M}(v)+d^+_{\mathcal{P}}(v))\nonumber\\
&\geq\frac{|S_1^\prime|-|S_2^\prime|+|S_{3, 1}^\prime|-|S_{3, 2}^\prime|+|V(\mathcal{P})|}{2}+k+|M_2|-d^+_{M}(v)-d^+_{\mathcal{P}}(v).
\end{align}
To further simplify inequalities $(\ref{313})$ and $(\ref{314})$, we first assert that $d^-_{M}(u)+d^+_{M}(v)\leq|V(M)|$.
Suppose, for contradiction, that $d^-_{M}(u)+d^+_{M}(v)\geq|V(M)|+1$. This implies the existence of an arc $xy \in M$, such that $d^-_{xy}(u)+d^+_{xy}(v)\geq3$.

\smallskip

If $xy \in M_1$, i.e., $xy$ is an arc from $S_{3, 1}^\prime$ to $S_1^\prime$, then $vx\notin A(D)$. Otherwise, by the definition of $S_{3, 1}^\prime$, for any vertex $w\in S_{3, 1}^\prime\setminus V(M)$, $d^-_{S_1^\prime}(w), d^+_{S_2^\prime}(w)>\frac{1}{2}(1-3\varepsilon^{1/3})n$, and for any vertex $v^\prime\in S_2^\prime$, $\delta^0_{S_2^\prime\setminus V(M)}(v^\prime)\geq(1-10\sqrt{\varepsilon})|S_2^\prime\setminus V(M)|$. This implies the existence of a vertex $v^\prime\in S_2^{\prime}\setminus V(M)$ such that $wv^\prime, vx\in A(D)$, leading to a $S_1^\prime$-path of length $4$ of the form $S_1^\prime\rightarrow w\rightarrow v^\prime \rightarrow xy$, contradicting the minimality of $|S_{3, 1}^\prime\setminus V(M)|$. Hence $vx\notin A(D)$, and then $xu, yu, vy\in A(D)$ since $d^-_{xy}(u)+d^+_{xy}(v)\geq3$. Replacing $xy$ in $M_1$ with $xu$ leads to a $S_1^\prime$-path of length $3$ of the form $S_1^\prime\rightarrow w\rightarrow v^\prime\rightarrow vy$, where $v^\prime\in N^+_{S_2^\prime\setminus V(M)}(w)\cap N^-_{S_2^\prime\setminus V(M)}(v)$, 
again contradicting the minimality of $|S_{3, 1}^\prime\setminus V(M)|$. 

\smallskip

Similarly, if $xy\in M_2$, assuming $xu\in A(D)$ leads to a $S_1^\prime$-path of length $4$ of the form $S_1^\prime\rightarrow w \rightarrow v^\prime v\rightarrow xu$, contradicting the minimality of $|S_{3, 1}^\prime\setminus V(M)|$. Therefore, $vx, vy, yu\in A(D)$ and $xu\notin A(D)$, implying another $S_1^\prime$-path of length $4$ of the form $S_1^\prime\rightarrow w\rightarrow v^\prime\rightarrow yu$, again a contradiction. Thus, we prove that $d_{xy}^-(u)+d_{xy}^+(v)\leq2$ for any $xy\in M$, which concludes that $d^-_{M}(u)+d^+_{M}(v)\leq|M|$.

\smallskip

By a similar argument to the one used for $M$, we can prove that $d^-_{\mathcal{P}}(u)+d^+_{\mathcal{P}}(v)\leq|V(\mathcal{P})|$. Specifically, by choosing $\mathcal{P}$ such that the remaining of $|S_{3, 1}^\prime|$ is minimized, we can derive this inequality through analogous reasoning. For brevity, we omit the detailed proof here.

\smallskip

So from inequalities (\ref{313}) and (\ref{314}), we have that
\begin{align*}\label{2m}
\begin{split}
|N^-_{S_2^\prime}(u)\cup N^+_{S_1^\prime}(v)|\geq|S_{3, 1}^\prime|-|S_{3, 2}^\prime|-|M|\geq r-|M|,
\end{split}
\end{align*}
implying the existence of $r-|M|$ disjoint arcs $vu$ from $S_2^\prime$ to $S_1^\prime$. Then for $r$ distinct vertices $w$ in $S_{3, 1}^\prime$, by the definition of $S_{3, 1}^\prime$, we know that $d^-_{S_1}(w), d^+_{S_2}(w)>\frac{1}{2}(1-2\varepsilon^{1/3})n$. Combining this with the properties of $S_2^\prime$, we have: $$|N_{S_2^\prime}^-(v)\cap N_{S_2^\prime}^+(w)|\geq(1-10\sqrt{\varepsilon})|S_2^\prime|+
(1-2\varepsilon^{1/3})n/2-|S_2^\prime|>n/3.$$
Thus, there exists an $S_1^\prime$-path of length $4$ of the form $S_1^\prime\rightarrow w\rightarrow S_2^\prime\rightarrow v\rightarrow u$. Similarly, we can construct $r-|M|$ disjoint $S_1^\prime$-paths of length $4$, utilizing a distinct arc $vu$ with $u\in S_1^\prime$ and $v\in S_2^\prime$. Additionally, for the  remaining $|M|$ distinct vertices in $S_{3, 1}^\prime$, we can obtain $|M|$ disjoint $S_1^\prime$-paths of the form $S_1^\prime\rightarrow M_1$ and $S_1^\prime$-paths of the form $M_2\rightarrow S_2^\prime$. This completes the proof of Claim \ref{S}.
\end{proof}

We now complete \textbf{Step 3}. Recall that $\mathcal{P}$ denotes the set of disjoint paths obtained from Claim \ref{ccccl}. Furthermore, for all $i\in[k]$ and $j\in[2]$, let $\mathcal{P}_j\subseteq\mathcal{P}$ be the collection of disjoint $(v_i, S_j)$-paths and $(S_j, v_i^\prime)$-paths. Without loss of generality, assume indices $i\in[l]$ correspond to paths of $\mathcal{P}_1$ and $i\in[k]\setminus[l]$ to paths of $\mathcal{P}_2$. Let $\mathcal{P}^\prime=\mathcal{P}^\prime_1\cup \mathcal{P}^\prime_2$ be the set of disjoint paths covering all vertices of $S_3^\prime$ from Claim \ref{S}, where for each $j\in[2]$, $\mathcal{P}^\prime_j$ denotes the disjoint $S_j^\prime$-paths, and $\mathcal{P}^\prime_1$ and $\mathcal{P}^\prime_2$ are disjoint. Define $S_j^{\prime\prime}=S_j^\prime\setminus V(\mathcal{P}^\prime)$ for $j\in[2]$. We proceed by case analysis:

\smallskip

\noindent \textbf{Case A1. $|V(S_1^{\prime\prime})\cup V(\mathcal{P}_1)\cup V(\mathcal{P}^\prime_1)|=\sum_{i\in[l]}n_i+l$.} By symmetry, this implies $|V(S_2^{\prime\prime})\cup V(\mathcal{P}_2)\cup V(\mathcal{P}^\prime_2)|=\sum_{i\in[k]\setminus[l]}n_i+k-l$. Property $(\mathcal{A}1)$ guarantees that for each $j\in[2]$, except for a subset $R_{j}\subseteq S_j^{\prime\prime}$ with $|R_{j}|\leq 10\sqrt{\varepsilon}|S_j|$, every vertex $v\in S_j^{\prime\prime}\setminus R_{j}$ satisfies $\delta^0_{S_{j}^{\prime\prime}}(v)\geq (1-20\sqrt{\varepsilon})|S_{j}^{\prime\prime}|$, while vertices $u$ in $R_{j}$ satisfy $\delta^0_{S_{j}^{\prime\prime}}(u)\geq\frac{\varepsilon^{1/3}|S_{j}^{\prime\prime}|}{2}$. This degree structure allows us to construct short disjoint $(v_i, u_i)$-paths and $(u_{i}^\prime, v_{i}^\prime)$-paths (for $i\in[k]$) within $D[S_1^{\prime\prime}]$ and $D[S_2^{\prime\prime}]$, respectively, utilizing:

$\bullet$ non-exceptional vertices of $S_1^{\prime\prime}$ and $S_2^{\prime\prime}$;

$\bullet$ disjoint paths from $\mathcal{P}\cup\mathcal{P}^\prime$;

$\bullet$ all vertices of $R_{1}$ and $R_{2}$.

\noindent Then, for convenience, let $S_1^0$ and $S_2^0$ denote the remaining vertices of $S_1^{\prime\prime}$ and $S_2^{\prime\prime}$, respectively. Then we get that

$\bullet$ $|S_j^0|\geq|S_j|-2\varepsilon^{1/6}n$ for $j\in[2]$, and

$\bullet$ for any $u\in S_j^0$, $\delta^0_{S_j^0}(u)\geq\frac{1}{2}(1-3\varepsilon^{1/6})n\geq(1-4\varepsilon^{1/6})|S_j^0|$.

\noindent Since the vertex pairs $(u_i, u_i^\prime)$ (for $i\in[k]$) lie in $S_j^{\prime\prime}\setminus R_j$, Proposition \ref{prop}, applies to the the $\varepsilon$-almost complete subdigraphs $D[S_1^0]$ and $D[S_2^0]$, yielding the desired $(u_i, u_{i}^\prime)$-subpaths ($i\in[k]$) of appropriate lengths. This completes Case \ref{case3.1}.

\smallskip

\noindent \textbf{Case A2. $|V(S_1^{\prime\prime})\cup V(\mathcal{P}_1)\cup V(\mathcal{P}^\prime_1)|<\sum_{i\in[l]}n_i+l$.} By symmetry, this implies $|V(S_2^{\prime\prime})\cup V(\mathcal{P}_2)\cup V(\mathcal{P}^\prime_2)|>\sum_{i\in[k]\setminus[l]}n_i+k-l$. We focus on the subcase where $V(P_1)\cap S_2=\emptyset$ for all $P_1\in \mathcal{P}_1$, and $V(P_2)\cap S_1=\emptyset$ for all $P_2\in \mathcal{P}_2$. Remaining subcases follow analogous reasoning and are omitted for brevity.

\smallskip

Since $\delta^0(D)\geq n/2+k$, for any vertex $u\in S_1$ and any $v\in S_2$:

$\bullet$ if $uv\in A(D)$, then $|N^+(u)\cap N^-(v)|\geq2\delta^0(D)-(n-1)\geq1+2k$;

$\bullet$ if $uv\notin A(D)$, then $|N^+(u)\cap N^-(v)|\geq 2\delta^0(D)-(n-2)\geq2+2k$.

\noindent Symmetrically, analogous bounds hold for $|N^-(u)\cap N^+(v)|$. These inequalities ensure: 

$\bullet$ two disjoint minimal $(S_1, S_2)$-paths of length $\leq2$, and

$\bullet$ two disjoint minimal $(S_2, S_1)$-paths of length $\leq2$.

\noindent Using these paths and following a procedure analogous to Case A1, we establish the validity of Case \ref{case3.1}.

For the case when $|V(S_1^{\prime\prime})\cup V(\mathcal{P}_1)\cup V(\mathcal{P}^\prime_1)|>\sum_{i\in[l]}n_i+l$ (by symmetry, this implies $|V(S_2^{\prime\prime})\cup V(\mathcal{P}_2)\cup V(\mathcal{P}^\prime_2)|<\sum_{i\in[k]\setminus[l]}n_i+k-l$), it is symmetric to Case A2. By swapping $S_1^{\prime\prime}$, $\mathcal{P}^\prime_1$ and $S_2^{\prime\prime}$, $\mathcal{P}^\prime_2$, respectively, the proof follows identically to that of Case A2, and is therefore omitted.
\end{proof}
\begin{Case}\label{case3.2}
$|U_1\cap U_2|>(1/2-\varepsilon_1)n$.
\end{Case}
Likewise, in this case, we complete the proof of Lemma \ref{case} in three steps:

\smallskip

\noindent\textbf{Step 1.} We prove that for every vertex pair $(v_i, v_i^\prime)$, by applying Definition \ref{definition3.3} with $(U_1, U_2)_{\ref{definition3.3}}=(W_2^\prime, W_4^\prime)$, the following hold (Claim \ref{C1}):\\
(\uppercase\expandafter{\romannumeral1}) if $n_i$ is even, then one of the following holds:

(i) $(v_i, v_i^\prime)\in V_1$ or $V_2$, or

(ii) $(v_i, v_i^\prime)\in V_{3}$ and there exists exactly one of

$~~~~$ $\bullet$ a $(v_i, W_2^\prime)$-path of length $2$ and $(W_{4}^\prime, v_i^\prime)$-path of length $2$, or

$~~~~$ $\bullet$ a $(v_i, W_{4}^\prime)$-path of length $3$ and $(W_2^\prime, v_i^\prime)$-path of length $3$, or

(iii) $(v_i, v_i^\prime)\in V_{4}$ and there exists exactly one of:

$~~~~$ $\bullet$ a $(v_i, W_4^\prime)$-path of length $2$ and $(W_{2}^\prime, v_i^\prime)$-path of length $2$, or

$~~~~$ $\bullet$ a $(v_i, W_{2}^\prime)$-path of length $3$ and $(W_4^\prime, v_i^\prime)$-path of length $3$.\\
\smallskip
(\uppercase\expandafter{\romannumeral2}) If $n_i$ is odd, then one of the following holds:

(iv) $(v_i, v_i^\prime)\in V_3$ or $V_4$, or

(v) $(v_i, v_i^\prime)\in V_{1}$ and there exists exactly one of:

$~~~~$ $\bullet$ a $(v_i, W_{2}^\prime)$-path of length $2$ and $(W_{2}^\prime, v_i^\prime)$-path of length $2$, or

$~~~~$ $\bullet$ a $(v_i, W_{4}^\prime)$-path of length $3$ and $(W_{4}^\prime, v_i^\prime)$-path of length $3$, or

(vi) $(v_i, v_i^\prime)\in V_{2}$ and there exists exactly one of:

$~~~~$ $\bullet$ a $(v_i, W_{4}^\prime)$-path of length $2$ and $(W_{4}^\prime, v_i^\prime)$-path of length $2$, or

$~~~~$ $\bullet$ a $(v_i, W_{2}^\prime)$-path of length $3$ and $(W_{2}^\prime, v_i^\prime)$-path of length $3$.

\smallskip
\noindent\textbf{Step 2.} We prove that there is a set $\mathcal{R}$ of disjoint paths such that $|W_2^\prime\setminus V(\mathcal{R})|=|W_4^\prime\setminus V(\mathcal{R})|$ (Claim \ref{D1}).

\smallskip

\noindent\textbf{Step 3.} In the balanced $\varepsilon$-almost complete bipartite subdigraph $(W_2^\prime\setminus V(\mathcal{R}), W_4^\prime\setminus V(\mathcal{R}))$, by Proposition \ref{prop} we obtain $2k$ disjoint paths with the desired lengths, where the initials and the terminals of these paths correspond to the required conditions. These paths, when combined with the disjoint paths in $\mathcal{R}$, form the desired $(v_i, v_i^\prime)$-paths for all $i\in[k]$.
\begin{proof}
We begin by performing some preparatory work before proceeding to \textbf{Step 1}. Clearly, in this case, we have that
$e^+(W_2^\prime, W_4^\prime), e^+(W_4^\prime, W_2^\prime)\geq|W_2^\prime|\cdot|W_4^\prime|-2\varepsilon n^2$, which implies that
$$|E_{4}\cap W_i^\prime|\leq|E_{3}\cap W_i^\prime|\leq\sqrt{10\varepsilon}|W_i^\prime|\ \mbox{for every}\ i\in\{2, 4\},$$ where the vertex sets $E_3$ and $E_4$ are defined in Definition \ref{definition3.2} with $(U_1, U_2)_{\ref{definition3.2}}=(W_2^\prime, W_4^\prime)$.

\smallskip

We first address the exceptional vertices of Type I$_4$ in $W_2^\prime\cup W_4^\prime$ using the following operation. For convenience, let $R=V(D)\setminus(F\cup W_2\cup W_4)$. If there exists a vertex $x$ in $E_{4}\cap W_2^\prime\cup R$ (resp., a vertex $y$ in $E_{4}\cap W_4^\prime\cup R$) such that $s_{W_4^\prime}(x)>\varepsilon^{1/3}|W_4^\prime|$ (resp., $s_{W_2^\prime}(y)>\varepsilon^{1/3}|W_2^\prime|$), then we move $x$ (resp., $y$) into the vertex set $W_4^\prime$ (resp., $W_2^\prime$) and update the sets $W_2^\prime$ and $W_4^\prime$. We repeat this operation until no such vertices $x$ and $y$ exist.

\smallskip

Next let $S_1=W_2^\prime\setminus E_4$ and $S_2=W_4^\prime\setminus E_4$, and let $S_3$ be the set of remaining vertices of $D$, that is $S_3=V(D)\setminus(F\cup S_1\cup S_2)$. Clearly, $|S_1|, |S_2|\geq(1/2-\varepsilon_1-\sqrt{10\varepsilon})n$. Together with $\delta^0(D)\geq n/2+k$ and the definitions of $S_1$, $S_2$ and $S_3$, we have the following properties:
\begin{itemize}\item[$(\mathcal{B}1)$] \emph{for every $i\in[2]$, apart from at most $10\sqrt{10\varepsilon}|S_i|$ exceptional vertices, all vertices in $S_i$ have strongly semi-degrees at least $(1-10\sqrt{\varepsilon})|S_{3-i}|$ in $S_{3-i}$, and the semi-degrees of these exceptional vertices are at least $\frac{\varepsilon^{1/3}|S_{3-i}|}{8}$ in $S_{3-i}$, and }
\item[$(\mathcal{B}2)$]
\emph{for every vertex $v\in S_3$, $\delta^0_{S_i}(v)\leq 2\varepsilon^{1/3}n$ for each $i\in[2]$, and either $d^+_{S_1}(v), d^-_{S_2}(v)>\frac{(1-2\varepsilon^{1/3})n}{2}$ or $d^-_{S_1}(v), d^+_{S_2}(v)>\frac{(1-2\varepsilon^{1/3})n}{2}$.}
\end{itemize}
Based on the semi-degree condition of $D$ and the cardinalities of $S_1$ and $S_2$, we know that for any vertex $v_i\in F$, $\delta^0_{S_1\cup S_2}(v_i)\geq\frac{n}{2}-k-|S_3|\geq\frac{1}{2}(1-2\varepsilon^{1/3})n$. Recall that $n_i$ is the length of the $(v_i, v_{i}^\prime)$-path for each $i\in[k]$. Without loss of generality, we assume that $n_1, \ldots, n_s$ are even, and $n_{s+1}, \ldots, n_k$ are odd.
By Definition \ref{definition3.3} with $(U_1, U_2)_{\ref{definition3.3}}=(S_1, S_2)$, for each $i\in[k]$,

$\bullet$ $(v_{i}, v_{i}^\prime)\in V_1$ if $|N^+(v_{i})\cap S_1|\geq 4k$ and $|N^-(v_{i}^\prime)\cap S_1|\geq 4k$;

$\bullet$ $(v_{i}, v_{i}^\prime)\in V_2$ if $|N^+(v_{i})\cap S_2|\geq 4k$ and $|N^-(v_{i}^\prime)\cap S_2|\geq 4k$;

$\bullet$ $(v_{i}, v_{i}^\prime)\in V_{3}$ if $|N^+(v_{i})\cap S_1|\geq 4k$ and  $|N^-(v_{i}^\prime)\cap S_2|\geq 4k$;

$\bullet$ $(v_{i}, v_{i}^\prime)\in V_{4}$ if $|N^+(v_{i})\cap S_2|\geq 4k$ and  $|N^-(v_{i}^\prime)\cap S_1|\geq 4k$.

\smallskip

Now we proceed to \textbf{Step 1}. For any vertex pair $(v_i, v_{i}^\prime)$ with $i\in[s]$, we consider the following cases:

$(C1)$ $(v_i, v_{i}^\prime)\in V_1$ or $V_2$, or

$(C2)$ $(v_{i}, v_{i}^\prime)\in V_{j+2}$ for some $j\in[2]$, and there exists exactly one arc from $N_{S_j}^+(v_{i})$ to $S_j$, or one arc from $S_{3-j}$ to $N_{S_{3-j}}^-(v_{i}^\prime)$. \\
If either $(C1)$ or $(C2)$ holds, then we do nothing. Symmetrically, for any vertex pair $(v_i, v_{i}^\prime)$ with $s+1\leq i\leq k$, we consider the following cases:

$(C3)$ $(v_{i}, v_{i}^\prime)\in V_{j+2}$ for some $j\in[2]$, or

$(C4)$ $(v_{i}, v_{i}^\prime)\in V_j$ for some $j\in[2]$, and there exists exactly an arc from $N_{S_j}^+(v_{i})$ to $S_j$ or an arc from $S_j$ to $N_{S_j}^-(v_{i}^\prime)$.\\
If either $(C3)$ or $(C4)$ holds, then we also do nothing. Additionally, we define the set of disjoint paths used by $(C1)$-$(C4)$ as $\mathcal{P}^\prime$. Otherwise, let


$\bullet$ $V_3^\prime$ (resp., $V_4^\prime$) be the set of vertex pairs $(v_i, v_{i}^\prime)$ with $i\in[s]$ that do not satisfy $(C1)$ and $(C2)$, and belong to $V_3$ (resp., $V_4$).

$\bullet$ $V_1^\prime$ (resp., $V_2^\prime$) be the set of vertex pairs $(v_i, v_{i}^\prime)$ with $s+1\leq i\leq k$ that do not satisfy $(C3)$ and $(C4)$, and are in $V_1$ (resp., $V_2$). \\
We then find a partition $(V_1^{\prime\prime}, V_2^{\prime\prime}, V_3^{\prime\prime}, V_4^{\prime\prime})$ of these vertex pairs, such that $V_i^{\prime\prime}\subseteq V_i^{\prime}$ for $i\in[4]$. Let $V^{\prime\prime}=V_1^{\prime\prime}\cup V_2^{\prime\prime}\cup V_3^{\prime\prime}\cup V_4^{\prime\prime}$, and define $l_i=|V_i^{\prime\prime}|$ with $i\in[4]$ and $l=l_1+l_2+l_3+l_4$. Further, we declare the following conclusion.
\begin{claim}\label{C1}
There exist $l$ disjoint paths of length $3$, denoted as $\mathcal{P}$, such that for any $P\in\mathcal{P}$, $|V(P)\cap S_3|=1$. Moreover, in $\mathcal{P}$:\\
$(i)$ there is a $(v_i, S_{3-j})$-path if $(v_i, v_{i}^\prime)\in V_{j+2}^{\prime\prime}$ for some $j\in[2]$ and,\\
$(ii)$ there is a $(v_{i}, S_{3-j})$-path if $(v_i, v_{i}^\prime)\in V_{j}^{\prime\prime}$.
\end{claim}
\begin{proof}
For any vertex $a_i\in N_{S_1\setminus\mathcal{P}^\prime}^+(v_i)$ with $(v_i, v_{i}^\prime)\in V_1^{\prime\prime}\cup V_3^{\prime\prime}$, since it does not satisfy $(C1)$-$(C4)$, we have that $d^+_{V^{\prime\prime}}(a_i)\leq 2l_1+l_2+l_3+2l_4$, and so $$d^+(a_i)\leq(2l_1+l_2+l_3+2l_4)+|S_2|+d^+_{S_3}(a_{i})+d^+_{\mathcal{P}^\prime}(a_i).$$ Similarly, for any vertex $b_j$ in $N^-_{S_2\setminus\mathcal{P}^\prime}(v_{j}^\prime)$ with $(v_j, v_{j}^\prime)\in V_2^{\prime\prime}$, we have that $$d^-(b_j)\leq(l_1+2l_2+l_3+2l_4)+|S_1|+d^-_{S_3}(b_j)+d^-_{\mathcal{P}^\prime}(b_j).$$ Clearly, $d^+_{\mathcal{P}^\prime}(a_i)+d^-_{\mathcal{P}^\prime}(b_j)\leq2(|F|-2l)$. On the other hand, by the lower bound of $\delta^0(D)$, we have that $$2(n/2+k)\leq d^+(a_{i})+d^+(b_{j}).$$ Together with $n=|S_1|+|S_2|+|S_3|+|F|$, this implies that
\begin{equation*}
\begin{split}
|N^+_{S_3}(a_i)\cap N^-_{S_3}(b_j)|\geq2k+l+l_3-l_4-|F|\geq l_1+l_2+2l_3.
\end{split}
\end{equation*}

\smallskip

Likewise, for any vertex $b_{i}\in N^+_{S_2\setminus \mathcal{P}^\prime}(v_{i})$ with $(v_i, v_{i}^\prime)\in V_2^{\prime\prime}\cup V_4^{\prime\prime}$, we have that $$d^+(b_{i})\leq(l_1+2l_2+2l_3+l_4)+|S_1|+d^+_{S_3}(b_{i})+d^+_{\mathcal{P}^\prime}(b_i).$$ Also for any vertex $a_{j}^\prime$ in
$N^-_{S_1\setminus V(\mathcal{P}^\prime)}(v_{j}^\prime)$ with $(v_j, v_{j}^\prime)\in V_1^{\prime\prime}$, we have that $$d^-(a_{j}^\prime)\leq(2l_1+l_2+2l_3+l_4)+|S_2|+d^-_{S_3}(a_{j}^\prime)+d^-_{\mathcal{P}^\prime}(a_{j}^\prime).$$ Obviously, $d^+_{\mathcal{P}^\prime}(b_i)+d^-_{\mathcal{P}^\prime}(a_{j}^\prime)\leq2(|F|-2l)$.
Combining with $\delta^0(D)\geq n/2+k$, $n=|S_1|+|S_2|+|S_3|+|F|$ and $l=l_1+l_2+l_3+l_4$, we get that
\begin{equation*}
\begin{split}
|N^+_{S_3}(b_{i})\cap N^-_{S_3}(a_{j}^\prime)|\geq 2k-|F|+l-l_3+l_4\geq l_1+l_2+2l_4.
\end{split}
\end{equation*}
Hence, with the help of the vertices in $S_3$, we can get construct the set $\mathcal{P}$ as desired. This completes the proof of Claim \ref{C1}.
\end{proof}
By Claim \ref{C1}, $(\mathcal{B}1)$ and the definitions of $\mathcal{P}$ and $\mathcal{P}^\prime$, along with the fact that semi-degree across vertices of $S_1\cup S_2$ is much larger than their order, we can obtain $2k$ disjoint minimal paths, called $\mathcal{P}^{\prime\prime}$, such that for every vertex pair $(v_i, v_{i}^\prime)$, in $\mathcal{P}^{\prime\prime}$ there is a path from $v_i$ to $u_i$ and a path from $u_{i}^\prime$ to $v_{i}^\prime$, where $u_i\in S_1$ and $u_{i}^\prime\in S_2$.

\smallskip

Next we proceed to \textbf{Step 2} (as outlined in Claim \ref{D1}). For every $i\in[2]$, set $S_i^\prime=S_i\setminus V((\mathcal{P}\cup \mathcal{P}^\prime\cup \mathcal{P}^{\prime\prime})\setminus\bigcup_{i=1}^k\{u_i, u_{i}^\prime\})$, and let $S_3^\prime$ be the remaining vertices of $S_3$. We now prove that the following conclusion holds.
\begin{claim}\label{D1}
There is a set $\mathcal{Q}$ of disjoint $(S_{2}^\prime, S_1^\prime)$-paths such that $|S_1^\prime\setminus V(\mathcal{Q})|=|S_2^\prime\setminus V(\mathcal{Q})|$, and $V(\mathcal{Q})\cap\bigcup_{i=1}^k\{u_i, u_{i}^\prime\}=\emptyset$.
\end{claim}
\begin{proof}
We can assume that $|S_1^\prime|\neq|S_2^\prime|$, since otherwise, setting $\mathcal{Q}=\emptyset$ satisfies the claim. Without loss of generality, suppose $|S_1^\prime|>|S_2^\prime|$. Let $F^\prime=V(D)\setminus(F\cup\bigcup_{i=1}^3 S_i^\prime)$. We denote

$\bullet$ $S_{3, 1}^\prime$ as the set of vertices of $S_3^\prime$ such that $d^-_{S_1}(v), d^+_{S_2}(v)>\frac{(1-2\varepsilon^{1/3})n}{2}$, and

$\bullet$ $S_{3, 2}^\prime$ as the set of vertices of $S_3^\prime$ such that $d^+_{S_1}(v), d^-_{S_2}(v)>\frac{(1-2\varepsilon^{1/3})n}{2}$.

\noindent Let $M_1$ (respectively, $M_2$ and $M_3$) be the maximum number of disjoint arcs in $S_1^\prime$ (respectively, from $S_{3, 1}^\prime$ to $S_1^\prime$ and from $S_1^\prime$ to $S_{3, 2}^\prime$), such that $M_1$, $M_2$ and $M_3$ are all disjoint. Let $M=\bigcup_{i=1}^3M_i$. Clearly, by $(\mathcal{B}1)$ and $(\mathcal{B}2)$, every arc in $M_1$ corresponds to a $(S_1^\prime, S_2^\prime)$-path of length $2$ of the form $M_1\rightarrow S_2^\prime$, and every arc in $M_2$ (resp., $M_3$) corresponds to a $(S_1^\prime, S_2^\prime)$-path of length $2$ of the form $S_1^\prime\rightarrow M_2\rightarrow S_2^\prime$ (resp., $M_3\rightarrow S_1^\prime\rightarrow S_2^\prime$). Let $\mathcal{Q}$ be the set of these disjoint $(S_1^\prime, S_2^\prime)$-paths. We choose $M$ to be as large as possible so that $|S_1^\prime\setminus V(\mathcal{Q})|-|S_2^\prime|$ is minimum.

\smallskip

If $|S_1^\prime\setminus V(\mathcal{Q})|-|S_2^\prime|\leq0$, then we are done by taking a subset of $\mathcal{Q}$. Otherwise, assume $|S_1^\prime\setminus V(\mathcal{Q})|-|S_2^\prime|>0$. We now prove that $|S_1^\prime\setminus V(\mathcal{Q})|=|S_2^\prime\setminus V(\mathcal{Q})|$. Suppose, for contradiction, that this is not the case. For a vertex $v$ in $S_1^\prime\setminus V(\mathcal{Q})$, by the degree condition  $d(v)\geq2\delta^0(D)=n+2k$, and noting that $d_{F\cup F^\prime}(v)\leq2|F|+|F^\prime|$ and $d_M(v)\leq|V(M)|$ (otherwise, we could replace an arc with a $3$-path, reducing the cardinality of $S_1^\prime$), we get:
\begin{equation*}
\begin{split}
d_{S_1^\prime}(v)&\geq n+2k-(|V(M)|+|S_{3, 1}^\prime\setminus V(M_1)|+|S_{3, 2}^\prime\setminus V(M_2)|+2|S_2^\prime|+2|F|+|F^\prime|)>0.
\end{split}
\end{equation*}
This implies that we can get another $(S_1^\prime, S_2^\prime)$-path $P$ such that $|S_1^\prime\setminus V(\mathcal{Q})|$ smaller for $\mathcal{Q}:=\mathcal{Q}\cup P$, leading to a contradiction.

\smallskip

Thus, we obtain a set $\mathcal{Q}$ of disjoint minimal $(S_2^\prime, S_1^\prime)$-paths and $|S_1^\prime\setminus V(\mathcal{Q})|=|S_2^\prime\setminus V(\mathcal{Q})|$ and $V(\mathcal{Q})\cap\bigcup_{i=1}^k\{u_i, u_{i}^\prime\}=\emptyset$. This proves Claim \ref{D1}.
\end{proof}
Finally we complete \textbf{Step 3.} Let $S_{3, i}^{\prime\prime}=S_{3, i}^{\prime}\setminus V(\mathcal{Q})$ for each $i\in[2]$ and $S_3^{\prime\prime}=S_{3}^{\prime}\setminus V(\mathcal{Q})$, where the set $\mathcal{Q}$ is obtained by Claim \ref{D1}. Then, by $(\mathcal{B}1)$ and $(\mathcal{B}2)$ we can construct two additional sets of disjoint paths, denoted as $\mathcal{Q}_1$ and $\mathcal{Q}_2$, covering all vertices of $S_{3}^{\prime\prime}$, such that

$\bullet$ $\mathcal{Q}_1$ is a set of disjoint $(S_1^\prime, S_2^\prime)$-paths, satisfying for any path $P_1\in \mathcal{Q}_1$, $$|V(P_1)\cap S_1^\prime|=|V(P_1)\cap S_2^{\prime}|=|V(P_1)\cap S_{3, 1}^{\prime\prime}|,$$

$\bullet$ $\mathcal{Q}_2$ is a set of disjoint $(S_2^\prime, S_1^\prime)$-paths, satisfying for any path $P_2\in\mathcal{Q}_2$, $$|V(P_2)\cap S_1^\prime|=|V(P_2)\cap S_2^\prime|=|V(P_2)\cap S_{3, 2}^{\prime\prime}|.$$ Next, by $(\mathcal{B}1)$, we can obtain $2k$ disjoint minimal paths with vertex pairs of end-vertices $\bigcup_{i=1}^{k}(u_i, w_{i}^\prime)\cup(w_i, u_{i}^\prime)$, where $w_i\in S_1^\prime$ and $w_{i}^\prime\in S_2^\prime$, such that these paths cover all vertices of $V(\mathcal{Q}\cup\mathcal{Q}_1\cup\mathcal{Q}_2)\cup V(\mathcal{P}\cup \mathcal{P}^\prime\cup \mathcal{P}^{\prime\prime})$ and the exceptional vertices of $S_1^\prime\cup S_2^\prime$.

\smallskip

Let $S_i^{\prime\prime}$ be the remaining vertices of $S_i^\prime$ with $i\in[2]$. Then $|S_i^\prime|\geq(1/2-\varepsilon^{1/6}/4)n$, and each vertex $u\in S_i^\prime$ has strong semi-degree of at least $(1-\varepsilon^{1/6})|S_{3-i}^\prime|$ in $S_{3-i}^\prime$. Finally, in the remaining balanced $\varepsilon$-almost complete bipartite subdigraph $(S_1^{\prime\prime}, S_2^{\prime\prime})$, by Proposition \ref{prop} we can construct all paths from $w_{i}^\prime$ to $w_i$ (for $i\in[k]$) with the required lengths. This completes the proof of Case \ref{case3.2}.
%
%
%
%
\end{proof}
\begin{Case}\label{subcase}
$\varepsilon_1n<|U_1\cap U_2|<(1/2-\varepsilon_1)n$.
\end{Case}
In this case, we also complete the proof of Lemma \ref{case} in three steps.

\smallskip

\noindent\textbf{Step 1.} Prove that for any vertex pair $(v_i, v_i^\prime)$ with $i\in[k]$ and any set $W_j^\prime$ with $j\in[4]$, there exists a $(v_i, W_j^\prime)$-path and a $(W_j^\prime, v_i^\prime)$-path, each of length at most $4$ (Claim \ref{asd}).

\smallskip

\noindent\textbf{Step 2.} Prove that there exist disjoint $W_1^\prime$-paths and $W_3^\prime$-paths to cover all vertices that have the small semi-degrees in $D[W_1^\prime]$ and $D[W_3^\prime]$, and there exist disjoint $(W_2^\prime, W_4^\prime)$-paths containing all vertices of $W_j^\prime$ for $j\in\{2, 4\}$ that have small strong semi-degrees in $W_{j+2}^\prime$, where $W_6^\prime=W_2^\prime$ (see Claim \ref{qwe}).

\smallskip

\noindent\textbf{Step 3.} In the remaining $\varepsilon$-almost complete subdigraphs and the balanced $\varepsilon$-almost complete bipartite subdigraph, we use Proposition \ref{prop} to obtain disjoint paths of the desired length with the specified initial and terminal.
\begin{proof}
We now begin the proof of this case. In particular, if $|W_2|\leq \varepsilon^{1/3} n$ or $|W_4|\leq \varepsilon^{1/3} n$, then the case reduces to Case \ref{case3.1} with $\varepsilon^{1/3}$ playing the role of $\varepsilon$. Symmetrically, if $|W_1|\leq \varepsilon^{1/3} n$ or $|W_3|\leq \varepsilon^{1/3} n$, then the case reduces to Case \ref{case3.2} by replacing $\varepsilon$ with $\varepsilon^{1/3}$. Hence in the following, for each $i\in[4]$, we assume that $|W_i|>\varepsilon^{1/3} n$, and set $W_i^\prime=W_i\setminus F$. It is easy to verify that for each $i\in\{1, 3\}$,
$$e(W_i^\prime)\geq|W_i^\prime|^2-3\varepsilon n^2,$$ which implies $$|E_2\cap W_i^\prime|\leq|E_1\cap W_i^\prime|\leq\sqrt{10\varepsilon}|W_i^\prime|,$$ where $E_1$ and $E_2$ are defined in Definition \ref{definition3.2} with $(U_1, U_2)_{\ref{definition3.2}}=(W_1^\prime, W_3^\prime)$.

\smallskip

Further, if there exists a vertex $x\in(E_2\cap W_1^\prime) \cup (V(D)\setminus(W_1\cup W_3\cup F))$ (resp., a vertex $y\in(E_2\cap W_3^\prime)\cup (V(D)\setminus(W_1\cup W_3\cup F))$) such that for each $\sigma\in\{+, -\}$, $d_{W_3^\prime}^\sigma(x)>\varepsilon^{1/3}|W_3^\prime|$ (resp., $d_{W_1^\prime}^\sigma(y)>\varepsilon^{1/3}|W_1^\prime|)$, then we move $x$ (resp., $y$) into the vertex set $W_3^\prime$ (resp., $W_1^\prime$) and update the vertex sets $W_1^\prime$ and $W_3^\prime$. We repeat this operation until no such vertices $x$ and $y$ exist. Next let $S_{1, 1}=W_1^\prime\setminus E_2$, $S_{1, 2}=W_3^\prime\setminus E_2$ and $S_{1, 3}=(W_1^\prime\cup W_3^\prime)\setminus (S_{1, 1}\cup S_{1, 2})$. Using the lower bound of $\delta^0(D)$ and the definitions $S_{1, i}$ with $i\in[3]$, we observe:
\begin{itemize}\item[$(\mathcal{C}1)$] \emph{for each $i\in[2]$,
there exists a subset $S_{1, i}^\prime\subseteq S_{1, i}$ with $|S_{1, i}^\prime|\leq 10\sqrt{\varepsilon}|S_{1, i}|$ such that\\
$\bullet$ for every vertex $v\in S_{1, i}\setminus S_{1, i}^\prime$, $\delta^0_{S_{1, i}}(v)\geq (1-10\sqrt{\varepsilon})|S_{1, i}|$, and\\
$\bullet$ for every vertex $v\in S_{1, i}^\prime$, $\delta^0_{S_{1, i}}(v)\geq\frac{\varepsilon^{1/3}|S_{1, i}|}{2}$.}
\end{itemize}
\noindent
\begin{itemize}\item[$(\mathcal{C}2)$] \emph{For every $v\in S_{1, 3}$, either $d^-_{S_{1, 1}}(v), d^+_{S_{1, 2}}(v)>(1-2\varepsilon^{1/3})|W_1^\prime|$ or $d^+_{S_{1, 1}}(v), d^-_{S_{1, 2}}(v)>(1-2\varepsilon^{1/3})|W_1^\prime|$. Also, $\delta^0_{S_{1, i}}(v)\leq\frac{\varepsilon^{1/3}|W_1^\prime|}{2}$ for each $i\in[2]$.}
\end{itemize}

Symmetrically, since $D[W_2^\prime\cup W_4^\prime]$ is an $\varepsilon$-almost complete bipartite digraph, we have that
$$e^+(W_2^\prime, W_4^\prime), e^+(W_4^\prime, W_2^\prime)\geq|W_2^\prime|\cdot|W_4^\prime|-2\varepsilon n^2,$$ which implies that $$|E_{4}\cap W_i^\prime|\leq|E_{3}\cap W_i^\prime|\leq\sqrt{10\varepsilon}|W_i^\prime|\ \mbox{for every}\ i\in\{2, 4\},$$ where $E_3$ and $E_4$ are defined in Definition \ref{definition3.2} with $(U_1, U_2)_{\ref{definition3.2}}=(W_2^\prime, W_4^\prime)$.

\smallskip

We first address the exceptional vertices of Type I$_4$ in $W_2^\prime\cup W_4^\prime$ using the following operation. If there exists a vertex $x\in(E_{4}\cap W_2^\prime)\cup (V(D)\setminus\cup_{i=1}^3 S_{1, i})$ (resp., a vertex $y\in(E_{4}\cap W_4^\prime)\cup (V(D)\setminus\cup_{i=1}^3 S_{1, i}$) such that $s_{W_4^\prime}(x)>\varepsilon^{1/3}|W_4^\prime|$ (resp., $s_{W_2^\prime}(y)>\varepsilon^{1/3}|W_2^\prime|$), then we move $x$ (resp., $y$) into the vertex set $W_2^\prime$ (resp., $W_1^\prime$) and update the sets $W_1^\prime$ and $W_2^\prime$. We repeat this operation until no such vertices $x$ and $y$ exist. Then let $S_{2, 1}=W_2^\prime\setminus E_4$ and $S_{2, 2}=W_4^\prime\setminus E_4$, and let $S_{2, 3}=(W_2^\prime\cup W_4^\prime)\setminus(S_{2, 1}\cup S_{2, 2})$ be the set of remaining vertices of $W_2^\prime\cup W_4^\prime$. Together with $\delta^0(D)\geq n/2+k$ and the definitions $S_{2, i}$ with $i\in[3]$, we have the following properties:
\begin{itemize}\item[$(\mathcal{C}3)$] \emph{for every $i\in[2]$, apart from at most $10\sqrt{10\varepsilon}|S_{2, i}|$ exceptional vertices, all vertices in $S_{2, i}$ have strongly semi-degrees of at least $(1-10\sqrt{\varepsilon})|S_{2, 3-i}|$ in $S_{2, 3-i}$, and the semi-degrees of these exceptional vertices are at least $\frac{\varepsilon^{1/3}|S_{2, 3-i}|}{8}$ in $S_{2, 3-i}$, and }
\item[$(\mathcal{C}4)$]
\emph{for every vertex $v\in S_{2, 3}$, either $d^+_{S_{2, 1}}(v),\ d^-_{S_{2, 2}}(v)>(1-2\varepsilon^{1/3})|W_2^\prime|$ or $d^-_{S_{2, 1}}(v),\  d^+_{S_{2, 2}}(v)$ $>(1-2\varepsilon^{1/3})|W_2^\prime|$. Also, $\delta^0_{S_{2, i}}(v)\leq 2\varepsilon^{1/3}|S_{2, i}|$ for each $i\in[2]$.}
\end{itemize}
Clearly, the vertex set satisfies: $V(D)\setminus F=S_{1, 1}\cup S_{1, 2}\cup S_{1, 3}\cup S_{2, 1}\cup S_{2, 2}\cup S_{2, 3}$. We also affirm the following assertion.
\begin{claim}\label{asd}
For any vertex pair $(v_i, v_{i}^\prime)$ with $i\in[k]$, there exists a path of length at most $4$ from $v_i$ to every subset $W_j^\prime$ \emph{(}$j\in[4]$\emph{)}. Similarly, there exists a path of length at most $4$ from each $W_j^\prime$ to $v_{i}^\prime$.
\end{claim}
\begin{proof}
By the pigeonhole principle, there is an integer $i_0$ with $i_0\in[4]$ such that $d^+_{W_{i_0}^\prime}(v_i)\geq\frac{|W_{i_0}^\prime|}{2}\geq n/8-4\sqrt{10\varepsilon}n$. Without loss of generality, assume $i_0=1$. By property $(A)$ in Definition \ref{ec} and Lemma \ref{claim1},
\begin{equation*}
\begin{split}
e^+(N^+_{W_1^\prime}(v_i), W_2^\prime)\geq|N^+_{W_1^\prime}(v_i)|\cdot|W_2^\prime|-\frac{\varepsilon^\prime n^2}{2}\geq\left(\frac{n}{8}-4\sqrt{10\varepsilon}n\right)\cdot|W_2^\prime|
-\frac{\varepsilon^\prime n^2}{2}.
\end{split}
\end{equation*}
The above inequality implies that at least half of $W_2^\prime$ (denoted as $X$) satisfies
$$d^-_{N^+_{W_1^\prime}(v_i)}(x)\geq n/16-4\sqrt{10\varepsilon}n,\ \mbox{for any}\ x\in X,$$ with $|X|\geq \frac{|W_2^\prime|}{2}\geq\frac{\varepsilon^{1/3}n}{2}$. Applying property $(A)$ again, $$e^+(X, W_3^\prime)\geq|X|\cdot|W_3^\prime|-\frac{\varepsilon^\prime n^2}{2}\geq\frac{\varepsilon^{2/3}n^2}{3}.$$ Thus, at least half of $W_3^\prime$ (denoted as $Y$) satisfies $$d^-_{U_2}(y)\geq \frac{|X|}{2}\geq\frac{\varepsilon^{1/3}n}{4},\ \mbox{for any}\ y\in Y.$$ Similarly, it satisfies that $e^+(Y, W_4^\prime)\geq|Y|\cdot|W_4^\prime|-\frac{\varepsilon^\prime n^2}{2}\geq\frac{\varepsilon^{2/3}n^2}{5}$, implying that at least $\frac{|W_4^\prime|}{2}\geq \frac{\varepsilon^{1/3}n}{2}$ vertices $w\in W_4^\prime$ satisfy $d^-_{Y}(w)\geq\frac{|Y|}{2}\geq\frac{\varepsilon^{1/3}n}{8}$.

\smallskip

Hence, Combining these results, we construct paths:

$\bullet$ $v_i\rightarrow W_1^\prime$: arc (length $1$).

$\bullet$ $v_i\rightarrow W_2^\prime$: $v_i\rightarrow N^+_{W_1^\prime}(v_i)\rightarrow X$ (length $2$).

$\bullet$ $v_i\rightarrow W_3^\prime$: $v_i\rightarrow N^+_{W_1^\prime}(v_i)\rightarrow X\rightarrow Y$ (length $3$).

$\bullet$ $v_i\rightarrow W_4^\prime$: $v_i\rightarrow N^+_{W_1^\prime}(v_i)\rightarrow X\rightarrow Y\rightarrow W_4^\prime$ (length $4$).

By analogous reasoning on the in-neighborhoods of $v_{i}^\prime$ (replacing out-neighbors with in-neighbors), paths from $W_i^\prime$ to $v_{i}^\prime$ exist with lengths matching the forward directions. This completes the proof.
\end{proof}
\emph{Remark}. Note that even after excluding the use of $\varepsilon^{1/2}n$ vertices, Claim \ref{asd} still holds. This implies that for any $i\in[k]$ and $j\in[4]$, there exists a $(v_i, W_j^\prime)$-path and a $(W_j^\prime, v_i^\prime)$-path of length at most $4$. We can therefore construct a set $\mathcal{P}^\prime$ of $2k$ disjoint paths (each of length $\leq4$) with the following properties for every vertex pair $(v_i, v_i^\prime)$ ($i\in[k]$):

(i) if $|W_1^\prime|, |W_3^\prime|\geq(1/2+\varepsilon)n$, then $\mathcal{P}^\prime$ contains a $(v_i, W_j^\prime)$-path and a $(W_j^\prime, v_i^\prime)$-path, for some $j\in\{1, 3\}$;

(ii) if $|W_2^\prime|, |W_4^\prime|\geq(1/2+\varepsilon)n$, then $\mathcal{P}^\prime$ contains a $(v_i, W_j^\prime)$-path and a $(W_{j+2}^\prime, v_i^\prime)$-path, for some $j\in\{2, 4\}$, where $W_6^\prime$ is identified with $W_2^\prime$ (i.e., indices cycle modulo $4$);

(iii) otherwise, $\mathcal{P}^\prime$ contains either:

$\bullet$  a $(v_i, W_j^\prime)$-path and a $(W_j^\prime, v_i^\prime)$-path for some $j\in\{1, 3\}$, or

$\bullet$ a $(v_i, W_j^\prime)$-path and a $(W_{j+2}^\prime, v_i^\prime)$-path for some $j\in\{2, 4\}$.

\noindent Crucially, the total vertex count satisfies $|V(\mathcal{P}^\prime)|\leq10k$.

\smallskip

For $i, j\in[2]$, define $S_{i, j}^\prime=S_{i, j}\setminus V(\mathcal{P}^\prime)$. We now establish the following key properties.
\begin{claim}\label{qwe}
We prove the following conclusions:

\textbf{\emph{(}E1\emph{)}} There exists a set $\mathcal{P}$ of disjoint $S_{1, 1}^\prime$-paths and $S_{1, 2}^\prime$-paths such that every vertex $u\in S_{1, 3}$ lies on a path in $\mathcal{P}$.

\textbf{\emph{(}E2\emph{)}} There exists a set $\mathcal{Q}$ of disjoint $(S_{2, 1}^\prime, S_{2, 2}^\prime)$-paths such that $|S_{2, 1}^\prime\setminus V(\mathcal{P}\cup \mathcal{Q})|=|S_{2, 2}^\prime\setminus V(\mathcal{P}\cup \mathcal{Q})|$.
\end{claim}
\begin{proof}
We first give the proof of (E1). For each $u\in S_{1, 3}$, by the property $(\mathcal{C}2)$, if $d^-_{S_{1, 1}^\prime}(u)>(1-3\varepsilon^{1/3})|W_1^\prime|$ and $d^+_{S_{1, 2}^\prime}(u)>(1-3\varepsilon^{1/3})|W_1^\prime|$, then Lemma \ref{claim1} and property $(A)$ of \textbf{EC1} imply that $$e^+(N^+_{S_{1, 2}^\prime}(u), S_{2, 2}^\prime)\geq |N^+_{S_{1, 2}^\prime}(u)|\cdot|S_{2, 2}^\prime|-\frac{\varepsilon^\prime n^2}{2}.$$ This yields at least $\frac{1}{2}|N^+_{S_{1, 2}^\prime}(v)|\geq\frac{1}{2}(1-4\varepsilon^{1/3})|S_{1, 2}^\prime|$ vertices $w\in S_{1, 2}^\prime$ with $$d^+_{S_{2, 2}^\prime}(w)\geq\frac{|S_{2, 2}^\prime|}{2}\geq\frac{\varepsilon^{1/3}n}{10}.$$ Reapplying property $(A)$ to $e^+(N^+_{S_{2, 2}^\prime}(w), S_{1, 1}^\prime)$, we obtain that $$e^+(N^+_{S_{2, 2}^\prime}(w), S_{1, 1}^\prime)\geq|N^+_{S_{2, 2}^\prime}(w)|\cdot|S_{1, 1}^\prime|-\frac{\varepsilon^\prime n^2}{2}.$$ Thus, there are at least $\frac{\varepsilon^{1/3} n}{20}$ disjoint paths of length $3$ and the form $u\rightarrow S_{1, 2}^\prime\rightarrow S_{2, 2}^\prime\rightarrow S_{1, 1}^\prime$. Since $d^-_{S_{1, 1}^\prime}(u)>(1-2\varepsilon^{1/3})|S_{1, 1}^\prime|$, there are at least $\varepsilon^{1/3} n$ arcs from $S_{1, 1}^\prime$ to $u$. Combining these arcs with the paths above, we construct at least $\frac{\varepsilon^{1/3} n}{20}$ disjoint $S_{1, 1}^\prime$-paths of length at most $4$ through $u$.

\smallskip

Similarly, if $d^+_{S_{1, 1}^\prime}(u)>(1-3\varepsilon^{1/3})|W_1^\prime|$ and $d^-_{S_{1, 2}^\prime}(u)>(1-3\varepsilon^{1/3})|W_1^\prime|$, analogous reasoning gives at least $\frac{\varepsilon^{1/3} n}{20}$ disjoint $S_{1, 2}^\prime$-paths of length at most $4$ through $u$: We can also prove that there exist at least $\frac{\varepsilon^{1/3} n}{20}$ disjoint paths of length $3$ and the form $u\rightarrow S_{1, 1}^\prime\rightarrow S_{2, 1}^\prime\rightarrow S_{1, 2}^\prime$, and since $d^-_{S_{1, 2}^\prime}(u)>(1-2\varepsilon^{1/3})|W_1^\prime|$, there are at least $\varepsilon^{1/3} n$ arcs from $S_{1, 2}^\prime$ to $u$.

\smallskip

Since $|S_{1, 3}|\leq \frac{\sqrt{10\varepsilon}n}{2}<\frac{\varepsilon^{1/3} n}{20}$, each $u\in S_{1, 3}$ can be assigned an $S_{1, 1}^\prime$-path or an $S_{1, 2}^\prime$-path, and all these paths are disjoint for distinct vertices $u$. Let $\mathcal{P}$ be the union of these paths. The total vertices in $\mathcal{P}$ satisfy $|V(\mathcal{P})|\leq 5\varepsilon^{1/2}n$, proving (E1).

\smallskip

We secondly give the proof of (E2). Define $S_{i, j}^{\prime\prime}=S_{i, j}^\prime\setminus V(\mathcal{P}\cup\mathcal{P}^\prime)$ for $i, j\in[2]$, where $\mathcal{P}^\prime$ is a set of pre-defined paths. By properties $(\mathcal{C}1)$-$(\mathcal{C}4)$, we have that $|S_{2, i}^{\prime\prime}|\geq\frac{\varepsilon^{1/3}n}{4}$. Assume $|S_{2, 1}^{\prime\prime}|\neq |S_{2, 2}^{\prime\prime}|$; otherwise, set $\mathcal{Q}=\emptyset$. Without loss of generality, assume $|S_{2, 1}^{\prime\prime}|>|S_{2, 2}^{\prime\prime}|$, and let $r=|S_{2, 1}^\prime|-|S_{2, 2}^\prime|$. Let $\mathcal{R}$ be a set of disjoint $S_{1, 1}^\prime$-paths, $S_{1, 2}^\prime$-paths and $(S_{2, 1}^\prime, S_{2, 2}^\prime)$-paths, with the property that for disjoint paths $P$ in $\mathcal{R}$, there is exactly one path of $\mathcal{P}$ that is a subgraph of $P$. Further we chose the set $\mathcal{R}$ such that the imbalance $r>0$ is minimized. Under this minimality condition, for at least $r$ vertices $u\in S_{2, 1}^{\prime\prime}$, we can assume that $$d^+_{S_{1, 1}^{\prime\prime}}(u)=0\ \mbox{and}\ d^-_{S_{1, 2}^{\prime\prime}}(u)=0.$$ Otherwise, if $d^+_{S_{1, 1}^{\prime\prime}}(u)>0$ or $d^-_{S_{1, 2}^{\prime\prime}}(u)>0,$ then there would exist an $S_{1, 1}^{\prime\prime}$-path  through $u$, or an $S_{1, 2}^{\prime\prime}$-path $P_2$ through $u$, respectively. Adding $P_i$ (for some $i\in[2]$) to $\mathcal{R}$ would reduce the imbalance $r$ contradicting the minimality of $r$.

\smallskip

Hence it follows from the lower of $\delta^0(D)$ that for each $u\in S_{2, 1}^{\prime\prime}$,
\begin{equation*}
\begin{split}
d_{S_{2, 2}^{\prime\prime}}(u)
\geq2\left(\frac{n}{2}+k\right)-|V(\mathcal{R})|-|S_{1, 1}^\prime\cup S_{1, 2}^\prime|-2|S_{2, 2}^{\prime}|\geq r,
\end{split}
\end{equation*}
where $d_{\mathcal{R}}(u)\leq |V(\mathcal{P})|+1$. Otherwise, $u$ could be inserted into some path in $\mathcal{R}$ to decrease $r$. Thus, $D[S_{2, 1}^{\prime\prime}]$ contains $r$ disjoint arcs, denoted as $\mathcal{M}$. By $(\mathcal{C}4)$, for any vertex $v\in S_{2, 3}$, either $d^+_{S_{2, 1}^{\prime\prime}}(v), d^-_{S_{2, 2}^{\prime\prime}}(v)>(1-3\varepsilon^{1/3})|S_{2, 1}^{\prime\prime}|$ or $d^-_{S_{2, 1}^{\prime\prime}}(v), d^+_{S_{2, 2}^{\prime\prime}}(v)>(1-3\varepsilon^{1/3})|S_{2, 1}^\prime|$. Combined with $(\mathcal{C}3)$, we can construct a set $\mathcal{Q}$
 of disjoint $(S_{2, 1}^\prime, S_{2, 2}^\prime)$-paths, covering $V(\mathcal{M})\cup S_{2, 3}$ such that $|S_{2, 1}^\prime\setminus V(\mathcal{Q})|=|S_{2, 2}^\prime\setminus V(\mathcal{Q})|$. This completes the proof of (E2).
\end{proof}

Let $\mathcal{P}^\prime$ be the set of disjoint paths consisting of $(v_i, W_j^\prime)$-paths and $(W_s^\prime, v_i^\prime)$-paths for every vertex pair $(v_i, v_i^\prime)$ with $i\in[k]$ and some $j, s\in[4]$, where $|V(\mathcal{P}^\prime)|\leq10k$. Define $S_{i, j}^{\prime\prime}=S_{i, j}^\prime\setminus V(\mathcal{P}\cup\mathcal{Q})$, where $\mathcal{P}$ and $\mathcal{Q}$ are the path sets obtained in Claim \ref{qwe}. Then by leveraging the properties $\mathcal{C}1$-$\mathcal{C}4$, we construct a set $\mathcal{P}^{\prime\prime}$ of $2k$ disjoint paths with end-vertex pairs $\bigcup_{i=1}^k(v_i, u_i)\cup (u_{i}^\prime, v_{i}^\prime)$ such that

\noindent$\bullet$ the residual sets satisfy $|S_{2, 1}^{\prime\prime}\setminus V(\mathcal{P}\cup \mathcal{P}^\prime\cup \mathcal{P}^{\prime\prime}\cup\mathcal{Q})|=|S_{2, 2}^{\prime\prime}\setminus V(\mathcal{P}\cup \mathcal{P}^\prime\cup \mathcal{P}^{\prime\prime}\cup\mathcal{Q})|$, and

\noindent$\bullet$ $\mathcal{P}^{\prime\prime}$ covers all vertices in $V(\mathcal{P}\cup \mathcal{P}^\prime\cup\mathcal{Q})\setminus\{u_i, u_i^\prime\}$ and all exceptional vertices of Type I$_1$ and Type I$_3$ in $S_{i, j}$ for any $i, j\in[2]$.

\smallskip

Let $S_{i, j}^0$ denote the set of remaining vertices in $S_{i, j}^{\prime\prime}$ for each $i, j\in[2]$. According to Claim \ref{qwe}, the remaining subdigraphs  $D[S_{1, j}^{\prime\prime}]$ ($j\in[2]$) are $\varepsilon$-almost complete, and $(S_{2, 1}^0, S_{2, 2}^0)$ is a balanced $\varepsilon$-almost complete bipartite digraph. Finally, by applying Proposition \ref{prop}, we can construct all required disjoint paths with specified lengths in the subdigraphs $D[S_{1, 1}^{\prime\prime}]$, $D[S_{1, 2}^{\prime\prime}]$ and $(S_{2, 1}^0, S_{2, 2}^0)$, where the end-vertices of these paths correspond to $\{u_i, u_i^\prime\}$ for all $i\in[k]$. This proves Case \ref{subcase}.
\end{proof}
Combining the results of Cases \ref{case3.1}, \ref{case3.2} and \ref{subcase}, we conclude that Lemma \ref{case} holds.
\end{proof}
\section{Concluding remarks}
The techniques developed in this work can be naturally extended to prove the following result. A detailed proof is omitted here but can be reconstructed through analogous arguments.
\begin{theorem}
Let $H$ be a digraph with $k$ arcs and $\delta(H)\geq1$. For any integers $n_1, \ldots, n_{k-2}$, there exist integers $n_{k-1}, n_k$ and constants $\alpha, \beta\in(0, 1)$ such that if $\max\{n_1, \ldots, n_k\}\leq n/2$ and $\sum_{n_i<\alpha n}n_i\leq\beta n$, then the following holds. There exists a constant $C_0$ such that if $D$ is a digraph of order $n\geq C_0k$ and $\delta^0(D)\geq n/2+k-1$, then $D$ is Hamiltonian $H$-linked, where the lengths of the subdivided paths are $n_1, \ldots, n_{k}$, respectively.
\end{theorem}

In this paper, we investigated the $H$-linkage problem in digraphs under a minimum semi-degree condition. A natural extension of this work is to consider the same problem in the context of robust outexpanders. Specifically, we propose the following open problem:
\begin{problem}
Let $H$ be any oriented digraph with $k$ arcs and $\delta(H)\geq1$. Does there exist a positive integer $C_0$ such that for any positive constants $\nu, \tau, \xi$ satisfying that $0<1/C_0\ll \nu\leq\tau\ll\xi<1$, the following holds? If $D$ is a digraph on $n\geq C_0k$ vertices and $\delta^0(D)\geq\xi n$ and $D$ is a robust $(\nu, \tau)$-outexpander, then for any injective map $f: V(H)\rightarrow
V(D)$ and any integer set $\mathcal{N}=\{n_1, \ldots, n_k\}$ satisfying that $n_i\geq C_0$ for each $i\in[k]$, there is a map $g: A(H)\rightarrow \mathcal{P}(D)$ such that for every arc $a_i=uv$, $g(a_i)$ is a directed path from $f(u)$ to $f(v)$ of length $n_i$, and different arcs are mapped into internally vertex-disjoint directed paths in $D$, and $\bigcup_{i\in[k]}V(g(a_i))=V(D)$.
\end{problem}

\end{document}